\documentclass[11pt]{amsart}
\usepackage{graphicx}
\usepackage{amsmath,amsthm}
\usepackage{amssymb}
\usepackage{pinlabel}
\usepackage[left = 3.5cm, right = 3.5 cm , top = 3 cm , bottom = 3cm ]{geometry}
\usepackage[utf8]{inputenc}
\usepackage{todonotes}

\newtheorem{thm}{Theorem}[section]

\newtheorem{lem}[thm]{Lemma}
\newtheorem{cor}[thm]{Corollary}

\newtheorem{notation}[thm]{Notation}

\newcommand{\cut}{\backslash\backslash}

\theoremstyle{remark}
\newtheorem{remark}[thm]{Remark}

\theoremstyle{definition}
\newtheorem{definition}[thm]{Definition}

\begin{document}

\title[Knot genus in a fixed 3-manifold]{The computational complexity of knot genus in a fixed 3-manifold}
\author{Marc Lackenby, Mehdi Yazdi}
\address{Mathematical Institute, University of Oxford, \newline Woodstock Road, Oxford OX2 6GG, United Kingdom}

\address{Department of Mathematics, King's College London, \newline Strand, London WC2R 2LS, United Kingdom}

\maketitle

\begin{abstract}
We show that the problem of deciding whether a knot in a fixed closed orientable 3-dimensional manifold bounds a surface of genus at most $g$ is in \textbf{co-NP}. This answers a question of Agol, Hass, and Thurston in 2002. Previously, this was known for rational homology 3-spheres, by the work of the first author. 
\end{abstract}

\section{Introduction}

Let $M$ be a compact 3-manifold, and let $K$ be a knot inside $M$. Since the work of Dehn in 1910 \cite{dehn1910topologie}, deciding whether $K$ can be unknotted has been a major question in low-dimensional topology. Dehn formulated the word and the isomorphism problems for groups in an attempt to solve this question. (The isomorphism problem was stated by Tietze \cite{tietze1908topologischen} in 1908 as well.) This in turn led to Novikov's discovery of the undecidability of the word problem for finitely presented groups \cite{novikov1955algorithmic} and the undecidability of the isomorphism problem for finitely presented groups by Adian \cite{Adian1957unsolvability} and Rabin \cite{rabin1958unsolvability}. Haken \cite{haken1961theorie} was the first person to prove that the unknot recognition problem is decidable using the theory of normal surfaces, introduced previously by Kneser \cite{kneser1929geschlossene}. 

Seifert defined the \emph{genus} of a knot in the 3-sphere \cite{seifert1935geschlecht}. Consider all connected, compact, embedded, orientable surfaces in $M$ whose boundary coincides with $K$, and let the genus, $g(K)$, be the minimum genus of the surfaces in this family. If there is no such surface, then we define $g(K)= \infty$ in this case. An easy observation is that $g(K) < \infty$ if and only if $K$ represents the trivial element in the first homology group $H_1(M; \mathbb{Z})$. Furthermore, $g(K) = 0$ if and only if $K$ is the unknot. 

Therefore, one of the most basic decision problems in low-dimensional topology is \textsc{\mbox{3-manifold} knot genus}, defined as follows: given a knot $K$ in a compact 3-manifold $M$ and a non-negative integer $g$, is the genus of $K$ less than or equal to $g$? The manifold $M$ is provided via a triangulation in which $K$ is a specified subcomplex. Agol, Hass and Thurston \cite{agol2006computational} proved that this problem is \textbf{NP}-complete. A consequence is that if \textsc{3-manifold knot genus} were to be in \textbf{co-NP}, then \textbf{NP} $=$ \textbf{co-NP}, contradicting a basic conjecture in complexity theory. 

It is natural to ask whether the difficulty of \textsc{3-manifold knot genus} is a consequence of the fact that $K$ and $M$ can both vary. What if we fix the manifold $M$, and only allow $K$ to vary? In \cite{agol2006computational}, Agol, Hass and Thurston asked about the computational complexity of this problem. The specific case where $M$ is the 3-sphere was addressed by the first author. He showed \cite[Theorem 1.3]{lackenby2016efficient} that, in this restricted setting, deciding whether a knot has genus less than or equal to $g$ \emph{is} in \textbf{co-NP}. More generally, if we are given a triangulation of a rational homology 3-sphere $M$, a knot $K$ as a subcomplex and an integer $g$, then the question `is $g(K)$ less than or equal to $g$?' lies in \textbf{co-NP}.

Let $N(K)$ be a tubular neighbourhood of $K$ with interior $N^\circ(K)$. The reason why knots in rational homology 3-spheres seem to be so much more tractable than in general 3-manifolds is that, in this situation, there can be only one possible homology class in $H_2(M - N^\circ(K), \partial N(K))$, up to sign, for a compact oriented spanning surface. This suggests that knots in more complicated 3-manifolds $M$ might be difficult to analyse, since as soon as $b_1(M) \geq 1$, there may be infinitely many possibilities for the homology class of a spanning surface. However, the main result of this paper is that knot genus in a \emph{fixed closed orientable 3-manifold} lies in \textbf{co-NP}. In particular, although knot genus in 3-manifolds is $\textbf{NP}$-hard, our result implies that, conditional on $\textbf{NP} \neq \textbf{co-NP}$, this is not the case when the ambient closed orientable 3-manifold is fixed.

In order to state this result more precisely, we need to explain how the knots $K$ in $M$ are presented. Any closed orientable 3-manifold is obtained by integral surgery on a framed link $L$ in the 3-sphere \cite{lickorish1962representation, wallace1960modifications}. When $M$ is closed, we fix such a surgery description of $M$, by fixing a diagram $D$ for $L$ where the framing of $L$ is diagrammatic framing and this specifies the surgery slopes.  We specify knots $K$ in $M$ by giving a diagram for $K \cup L$ that contains $D$ as a sub-diagram. The \emph{total number of crossings} of $K$ is defined as the number of crossings in this diagram between $K$ and itself and between $K$ and $L$.

\medskip
\noindent \textbf{Problem}: \textsc{Knot genus in the fixed closed orientable 3-manifold $M$}.\\ 
\emph{Input}: A diagram of $K \cup L$ that contains $D$ as a subdiagram, and an integer $g \geq 0$ in binary.\\
\emph{Input size}: Sum of the number of digits of $g$ in binary and the total number of crossings of $K$.\\  
\emph{Question}: Is the genus of $K$ less than or equal to $g$?\\ 

Strictly speaking, there are infinitely many decision problems here, one for each \mbox{3-manifold} $M$ and surgery diagram $D$.

\begin{thm}
Let $M$ be a closed, orientable 3-manifold given by integral surgery on a framed link in the 3-sphere. The problem \textsc{Knot genus in the fixed closed orientable 3-manifold $M$} lies in \textbf{co-NP}.
\label{main}
\end{thm}

This can be generalised to compact orientable 3-manifolds with non-empty toroidal boundary, as follows. Any compact orientable 3-manifold $M$ with toroidal boundary can be specified by means of the disjoint union of a link $\Gamma$ and a framed link $L$ in the 3-sphere. The manifold $M$ is obtained from $S^3$ by removing an open regular neighbourhood of $\Gamma$ and performing surgery along $L$. We fix a diagram $D$ for $\Gamma \cup L$, where again the surgery slopes on $L$ agree with the diagrammatic framing. We can then specify a knot $K$ in $M$ by giving a diagram for $K \cup \Gamma \cup L$ that contains $D$ as a sub-diagram. Again, the \emph{total crossing number} of $K$ is the number of crossings in this diagram between $K$ and itself and between $K$ and $\Gamma \cup L$. We say that \textsc{Knot genus in the fixed 3-manifold $M$} is the decision problem asking whether the genus of $K$ is less than or equal to a given non-negative integer.

\begin{thm}
Let $M$ be a compact, orientable 3-manifold with toroidal boundary given as above. The problem \textsc{Knot genus in the fixed 3-manifold $M$} lies in \textbf{co-NP}.
\label{main:boundary}
\end{thm}

For a connected orientable surface $S$, define the \emph{negative part of the Euler characteristic} as
\[\chi_-(S): = \max \{ 0 , -\chi(S) \}. \]
If $S$ has multiple components, define $\chi_-(S)$ as the sum of the corresponding values for the components of $S$.

\begin{definition}[Thurston complexity $c_\mathrm{Th}(K)$ for a knot $K$]
	Let $M$ be a compact orientable 3-manifold, and $K$ be a homologically trivial oriented knot in $M$. Set $X := M - N^\circ(K)$ as the complement of a tubular neighbourhood of $K$ in $M$. Denote by $\ell $ the longitude of $K$, which is defined as the boundary of any Seifert surface for $K$ in $X$. Define the \emph{Thurston complexity $c_\mathrm{Th}(K)$ for $K$} as 
	\begin{align*}
		c_\mathrm{Th}(K): = \min \{ \chi_-(S) \hspace{1mm} | \hspace{1mm} & S \text{ is a compact oriented properly embedded surface in }X \\ & \text{ with } [\partial S] = [\ell] \in H_1(\partial X ; \mathbb{Z})   \}.
	\end{align*}
	\label{def:Thurston complexity}
\end{definition}

We consider the decision problem \textsc{Thurston complexity for a knot in the fixed 3-manifold $M$} defined as follows. Here $M$ is compact and orientable, and can have (possibly non-toroidal) boundary. We can construct such manifolds by removing an open regular neighbourhood of a \emph{graph} $\Gamma$ from $S^3$ and then performing surgery on a link $L$ in the complement of $\Gamma$. Thus, we specify $M$ by means of a diagram $D$ for $\Gamma \cup L$, where again the surgery slopes on $L$ are given by diagrammatic framing.

\medskip
\noindent \textbf{Problem}: \textsc{Thurston complexity for a knot in the fixed 3-manifold $M$}.\\ 
\emph{Input}: A diagram of $K \cup \Gamma \cup L$ that contains $D$ as a subdiagram, and an integer $g \geq 0$ in binary.\\
\emph{Input size}: Sum of the number of digits of $g$ in binary and the total number of crossings of $K$.\\  
\emph{Question}: Is $c_\mathrm{Th}(K)$ less than or equal to $g$?\\ 

It is easy to show that when $M$ has toroidal boundary, $c_\mathrm{Th}(K)$ and $g(K)$ determine each other (Lemma \ref{computing genus from Thurston norm}). Theorem \ref{main:boundary} will then be a consequence of the following result.

\begin{thm}
	Let $M$ be a compact, orientable 3-manifold given by a fixed diagram $D$ for $\Gamma \cup L$, where $M$ is obtained from $S^3$ by removing an open regular neighbourhood of the graph $\Gamma$ and performing surgery on the framed link $L$. The problem \textsc{Thurston complexity for a knot in the fixed 3-manifold $M$} lies in \textbf{co-NP}.
	\label{thm:Thurston complexity for a knot}
\end{thm}



\subsection{Ingredients of the proof}

\medskip
(1) One of the key technical tools in the paper is the use of different measures of complexity for various objects. We introduce the relevant terminology now. 

For an integer $n$, let $C_{\mathrm{una}}(n) = |n|$ be the \emph{unary complexity} of $n$, and let $C_{\mathrm{dig}}(n)$ be the number of digits of $n$ when expressed in binary. In the case of negative $n$, we view the minus sign at the front as an extra digit. For a list of integers $(n_1, \cdots, n_k)$, let $C_{\mathrm{una}}(n_1, \cdots, n_k)$ be $\sum_i C_{\mathrm{una}}(n_i)$. Similarly, let $C_{\mathrm{dig}}(n_1, \cdots, n_k)$ be $\sum_i C_{\mathrm{dig}}(n_i)$. 
For a matrix $A$ with integer entries $A_{ij}$, let $C_{\mathrm{una}}(A)$ be $\sum_{ij} C_{\mathrm{una}}(A_{ij})$ and let $C_{\mathrm{dig}}(A)$ be $\sum_{ij} C_{\mathrm{dig}}(A_{ij})$. 
For a rational number $p/q$, with $p$ and $q$ in their lowest terms, let $C_{\mathrm{una}}(p/q)= C_{\mathrm{una}}(p) + C_{\mathrm{una}}(q)$ and let $C_{\mathrm{dig}}(p/q)= C_{\mathrm{dig}}(p) + C_{\mathrm{dig}}(q)$. 

The $C_{\mathrm{dig}}$ notions of size are the most natural ones and the ones that are most widely used in complexity theory, because they reflect the actual amount of memory required to store the number, list or matrix. However, we will also find the $C_{\mathrm{una}}$ versions useful.

\medskip
(2) Given a compact orientable 3-manifold $M$, the \emph{Thurston norm} is a semi-norm on $ H_2(M, \partial M; \mathbb{R})$, which is closely related to the notion of knot genus \cite{thurston1986norm}. See Section \ref{Sec:ThurstonNorm} for the precise definition. One of the main ingredients is the following result, proved by the first author in  \cite{lackenby2016efficient}.

\begin{thm}[Lackenby]
\textsc{Thurston norm of a homology class} is in \textbf{NP}.
\label{lackenby}
\end{thm}

The decision problem \textsc{Thurston norm of a homology class} takes as its input a triangulation $\mathcal{T}$ for a compact orientable 3-manifold $M$, a simplicial 1-cocycle $c$ and an integer $g$, and it asks whether the Thurston norm of the dual of $c$ is equal to $g$. The measure of complexity of $\mathcal{T}$ is its number of tetrahedra, denoted $|\mathcal{T}|$. The measure of complexity of $c$ is $C_{\mathrm{dig}}(c)$, where we view $c$ as a list of integers, by evaluating it against all the edges of $\mathcal{T}$ (when they are oriented in some way). The measure of complexity of $g$ is $C_{\mathrm{dig}}(g)$.

\medskip
(3) Thus, one can efficiently certify the Thurston norm of the dual of a \emph{single} cohomology class. However, in principle, a minimal genus Seifert surface for the knot $K$ could be represented by one of infinitely many classes. To examine all possible classes simultaneously, one needs a good picture of the Thurston norm ball. Thurston showed that the unit ball of this norm is a, possibly non-compact, convex polyhedron \cite{thurston1986norm}. More precisely, there is a linear subspace $W \subset H_2(M, \partial M; \mathbb{R})$ consisting of homology classes of norm zero such that the unit ball of the induced norm on $H_2(M, \partial M; \mathbb{R})/W$ is a compact convex polyhedron. In what follows $b_1$ denotes the first Betti number, and $H_2(M, \partial M; \mathbb{R})$ is identified with $H^1(M; \mathbb{R})$ using Poincar\'{e} duality.

\medskip
\noindent \textbf{Problem}: \textsc{Thurston norm ball for 3-manifolds with $b_1 \leq B$}.\\ 
\emph{Input}: A triangulation $\mathcal{T}$ of a compact orientable 3-manifold $X$ with $b_1(X) \leq B$, and a list of simplicial integral cocycles $\phi_1, \cdots, \phi_b$ that form a basis for $H^1(X; \mathbb{R})$. \\
\emph{Input size}: $|\mathcal{T}| + \sum_i C_{\mathrm{una}}(\phi_i)$.\\
\emph{Output}: The output is all the information that one needs to compute the Thurston norm ball: 
\begin{enumerate}
	\item a collection of integral cocycles that forms a basis for the subspace $W$ of $H^1(X; \mathbb{R})$ with Thurston norm zero;
	\item a list $V \subset H^1(X ; \mathbb{Q})$ of points that project to the vertices of the unit ball of $H^1(X; \mathbb{R})/W$, together with a list of subsets of these vertices that form faces. The elements of $V$ are given as rational linear combinations of $\phi_1, \cdots, \phi_b$.
\end{enumerate} 

\begin{remark}
	Note that we have used the unary notion of complexity for the size of the input here. Note also that $X$ is not assumed to have toroidal boundary.
\end{remark}

\begin{thm}
Fix an integer $B \geq 0$. The problem \textsc{Thurston norm ball for 3-manifolds with $b_1 \leq B$} lies in \textbf{FNP}, where $b_1$ denotes the first Betti number.
\label{thurston ball}
\end{thm}

Recall that \textbf{FNP} is the generalisation of \textbf{NP} from decision problems (where a yes/no answer is required) to function problems (where more complicated outputs might be required). A formal definition is given in Section \ref{Sec:Complexity}. 

It was known through the work of Tollefson and Wang \cite{TollefsonWang} that there is a (deterministic) algorithm to compute the unit ball of the Thurston norm using normal surfaces. Cooper and Tillmann \cite{cooper2009thurston}, and Cooper, Tillmann and Worden \cite{cooper2021thurston} gave another algorithm to compute the Thurston norm. However the computational complexity of these algorithms are not discussed in their work. 

At first sight, Theorem \ref{thurston ball} seems to lead easily to the proof of Theorem \ref{main}. However, its power is blunted by the unary notion of complexity that it uses for its input size. Thus, it only works well when $C_{\mathrm{una}}(\phi_i)$ is `small' for each $i$. That such a collection of simplicial cocycles exists in our setting is a consequence of the following surprising result.

\medskip
(4) Constructing an efficient basis for the second homology of a knot complement, for a fixed ambient manifold. 

Here, we work in the setting of compact orientable 3-manifolds with (possibly non-toroidal) boundary.

\begin{thm}
Let $M$ be a compact orientable 3-manifold given by removing an open regular neighbourhood of a (possibly empty) graph $\Gamma$ in $S^3$ and performing integral surgery on a framed link $L$ in the complement of $\Gamma$. Let $D$ be a fixed diagram for $\Gamma \cup L$ where the surgery slopes on $L$ coincide with the diagrammatic framing. Let $K$ be a homologically trivial knot in $M$, given by a diagram of $K \cup \Gamma \cup L$ that contains $D$ as a sub-diagram. Let $c$ be the total crossing number of $K$. Set $X = M - N^\circ(K)$ as the exterior of $K$ in $M $. There is an algorithm that builds a triangulation of $X$ with $O(c)$ tetrahedra, together with simplicial $1$-cocycles $\phi_1 , \cdots , \phi_b$ that form an integral basis for $H^1(X ; \mathbb{Z})$ with $\sum_i C_{\mathrm{una}}(\phi_i)$ at most $O(c^4)$. Moreover, the algorithm extends this triangulation of $X$ to a triangulation of $M$ with $O(c)$ tetrahedra, in which $K$ is simplicial. The algorithm runs in time polynomial in $c$. All the above implicit constants depend on the manifold $M$ and not the knot $K$. 
\label{basis}
\end{thm}

\medskip
(5) Controlling the number of faces and vertices of the Thurston norm ball polyhedron, in the presence of an efficient basis for the second homology.

A crucial step in the proof of Theorem \ref{thurston ball} is to bound the number of vertices and faces of the Thurston norm ball of the manifold $X$. The following result gives this, assuming that we have a good bound on the Thurston norm of a collection of surfaces that form a basis for $H_2(X, \partial X; \mathbb{R})$.

\begin{thm}
Let $X$ be a compact orientable 3-manifold, and let $m$ be a natural number. Assume that there exist properly immersed oriented surfaces $S_1 , \cdots, S_b$ in $X$ such that their homology classes form a basis for $H_2(X, \partial X; \mathbb{R}) $, and for each $1 \leq i \leq b$ we have $|\chi_-(S_i)| \leq m$. Denote by $W$ the subspace of $H^1(X  ; \mathbb{R})$ with trivial Thurston norm. 
The number of facets of the unit ball for the induced Thurston norm on $H^1(X  ; \mathbb{R})/W$ is at most $(2m+1)^b$, where $b = b_1(X)$ is the first Betti number of $X$. 
Hence, the number of vertices is at most $(2m+1)^{b^2}$ and the number of faces is at most $b(2m+1)^{b^2}$.
\label{number of faces}
\end{thm}
The proof of Theorem \ref{number of faces} uses the fact, due to Thurston \cite{thurston1986norm}, that the vertices of the dual unit ball of the Thurston norm are integral. See Theorem \ref{lattice points} for a result that gives an upper bound on the number of these integral points. In \cite{cooper2009thurston}, Cooper and Tillmann gave an exponential upper bound for the number of vertices of the Thurston norm ball as a function of the number of tetrahedra in a triangulation. This implies an exponential type upper bound on the number of faces. In the absence of an upper bound for the first Betti number, the total number of faces of the Thurston norm ball can in fact be exponentially large in terms of the number of tetrahedra. For example one can take the 3-manifold $M_n$ to be the complement of a link $L_n$ in the 3-sphere, where $L_n$ is the union of $n$ copies $K_1, \cdots, K_n$ of the same fixed non-trivial knot $K$ and with $K_i$ lying in disjoint balls. Then $M_n$ has a triangulation with $O(n)$ tetrahedra, and the number of faces of the Thurston norm ball of $M_n$ is $3^n+1$.




\medskip
(6) Constructing a basis for the subspace of the second homology with trivial Thurston norm.

\begin{thm}
\label{Thm:BasisForW}
Let $\mathcal{T}$ be a triangulation of a compact orientable irreducible 3-manifold $X$. If $X$ has any compressible boundary components, suppose that these are tori. Then there is a collection $w_1, \cdots, w_r$ of integral cocycles that forms a basis for the subspace $W$ of $H^1(X; \mathbb{R})$ consisting of classes with Thurston norm zero with $\sum_i C_{\mathrm{dig}}(w_i)$ at most $O(|\mathcal{T}|^3)$.
\end{thm}

This is proved by showing that $W\cap H^1(X ; \mathbb{Z})$ is spanned by fundamental normal surfaces, which is a consequence of work of Tollefson and Wang \cite{TollefsonWang}.


\medskip
(7) In Theorem \ref{Thm:BasisForW}, it is assumed that $X$ is irreducible and that every component of $\partial X$ is toroidal or incompressible. In Section \ref{Sec:SpheresDiscs}, we explain how we may ensure this. We cut along a maximal collection of compression discs and essential normal spheres to decompose $X$ into pieces, and we construct a new simplicial basis for the cohomology of the pieces. We also use the following result from \cite{lackenby2016efficient}.


\begin{thm}[Lackenby]
\label{Thm:IrreducibleNP}
The following decision problem lies in \textbf{NP}. The input is a triangulation of a compact orientable 3-manifold $M$ with (possibly empty) toroidal boundary and $b_1(M) > 0$, and the problem asks whether $M$ is irreducible.
\end{thm}

\begin{cor}

\label{Cor:IrredIncompNP}
The following decision problem lies in \textbf{NP}. The input is a triangulation of a compact orientable 3-manifold $M$ with $b_1(M) > 0$, and the problem asks whether $M$ is irreducible and has incompressible boundary.
\end{cor}

This is an immediate consequence of Theorem \ref{Thm:IrreducibleNP}. This is because a compact orientable 3-manifold $M$ is irreducible with incompressible boundary if and only if its double $DM$ is irreducible. This follows from the equivariant sphere theorem \cite{MeeksSimonYau}. Moreover, assuming that $M$ has no sphere boundary components, $b_1(DM)> 0$ if and only if $b_1(M)>0$.

\subsection{Varying $M$ and $K$}

As mentioned above, it seems very unlikely that Theorems \ref{main} and \ref{main:boundary} remain true if $M$  and $K$ are allowed to vary, because of the following result of Agol, Hass and Thurston \cite{agol2006computational}.

\begin{thm}
The following problem is \textbf{NP}-complete. The input is a triangulation of a closed orientable 3-manifold $M$, a knot $K$ in its 1-skeleton and an integer $g$, and the problem asks whether the genus of $K$ is at most $g$.
\end{thm}

However, what if we allow $M$ to vary but fix $b_1(M)$ in advance? It is unclear to the authors whether the problem of knot genus in such manifolds $M$ is likely to lie in \textbf{co-NP}. 

We believe that in this more general setting, Theorem \ref{basis} does not hold. Certainly, the proof of Theorem \ref{basis} required $M$ to be fixed. This bound on $\sum_i C_{\mathrm{una}}(\phi_i)$ was used to bound $\chi_-(S_i)$, where $S_i$ is a representative surface for the Poincar\'e dual of $c_i$. In the absence of such a bound, it is not clear that one can find a good upper bound on the number of faces and vertices of the Thurston norm ball for $H^1(X; \mathbb{R})$. In particular, it is an interesting question whether there is a sequence of 3-manifolds $X$ with bounded first Betti number and triangulations $\mathcal{T}$, where the number of vertices of the Thurston norm ball of $X$ grows faster than any polynomial function of $|\mathcal{T}|$.

\subsection{3-manifolds with non-toroidal boundary}

In Theorem \ref{main:boundary}, we assumed that $M$ has toroidal boundary. However, it is natural to consider a more general situation where $M$ may have higher genus boundary components. We conjecture that the generalisation of the problem \textsc{Knot genus in the fixed 3-manifold $M$} to these 3-manifolds $M$ lies in \textbf{co-NP}.

\subsection*{Acknowledgement} We would like to thank the referee for their careful reading of the article, and for their suggestions that greatly improved the article.

\section{Preliminaries}
\begin{notation}

For a subset $A$ of a topological space $Y$, the interior of $A$ is denoted by $A^ \circ$.

The first Betti number of a manifold $M$ is indicated by $b_1(M)$.
\label{notation}
\end{notation}	

\subsection{Complexity Theory}
\label{Sec:Complexity}

The material in this section is borrowed from \cite{arora2009computational,rich2008automata}, and we refer the reader to them for a more thorough discussion. 

Let $\{ 0 ,1 \}^*$ be the set of all finite strings in the alphabet $\{ 0 , 1 \}$. A \emph{problem} $P$ is defined as a function from $\{ 0 , 1 \} ^*$ to $\{ 0 ,1 \}^*$. Here the domain is identified with the \emph{inputs} or \emph{instances}, and the range is identified with the \emph{solutions}. A \emph{decision problem} is a problem whose range can be taken to be $\{ 0 ,1 \} \subset \{ 0 ,1 \}^*$. Intuitively a decision problem is a problem with yes or no answer.

A (deterministic) \emph{Turing machine} is a basic computational device that can be used as a model of computation. We refer the reader to Page 12 of \cite{arora2009computational} for a precise definition. By an \emph{algorithm} for the problem $P$, we mean a Turing machine $M$ that given any instance $I$ of the problem on its tape, computes and halts exactly with the solution $P(I)$.   We say $M$ runs in time $T: \mathbb{N} \longrightarrow \mathbb{N}$, if for any instance $I$ of binary length $|I|$, if we start the Turing machine $M$ with $I$ on its tape, the machine halts after at most $T(|I|)$ steps. 

The \emph{complexity class} \textbf{P} consists of all decision problems $P$ for which there exists a Turing machine $M$ and positive constants $c,d$ such that $M$ answers the problem in time $c  n^d$, where $n$ is the size of the input. 

The complexity class \textbf{NP} consists of decision problems such that their yes solutions can be efficiently \emph{verified}. By this we mean that there is a Turing machine that can verify the yes solutions in polynomial time. This is possibly a larger complexity class than the class \textbf{P}, which was described as the set of decision problems that can be efficiently \emph{solved}. In other words, \textbf{P} $\subseteq$ \textbf{NP}. The precise definition is as follows. By a \emph{language}, we mean a subset of $\{ 0 ,1 \}^*$. In our context, we have a decision problem $P \colon \{ 0 ,1 \}^* \longrightarrow \{ 0 ,1 \}$ and we take $L$ as the set of instances whose solutions are equal to $1$ (yes answer).

\begin{definition}
A language $L \subset \{ 0 ,1 \}^*$ is in \textbf{NP}, if there exists a polynomial $p \colon \mathbb{N} \longrightarrow \mathbb{N}$ and a Turing machine $M$ that runs in polynomial time (called the \emph{verifier} or \emph{witness} for $L$) such that for every instance $x \in \{ 0 ,1 \}^*$
\[ x \in L \iff \exists u \in \{ 0,1 \}^{p(|x|)} \hspace{2mm} \text{such that} \hspace{2mm} M(x,u)=1. \]
If $x \in L$ and $u \in \{ 0 ,1 \}^{p(|x|)}$ and $M(x,u)=1$, we call $u$ a \emph{certificate} for $x$.

A language $L \subset \{ 0 ,1 \}^*$ is in \textbf{co-NP} if its complement $\{ 0,1 \}^* \setminus L$ is in \textbf{NP}. 
\end{definition} 

Hence \textbf{co-NP} is the set of decision problems such that their no solutions can be efficiently verified. A decision problem is called \textbf{NP}-\emph{hard} if it is at least as hard as any other problem in \textbf{NP}. More specifically, every problem in \textbf{NP} is \emph{Karp-reducible} to any \textbf{NP}-hard problem. (See Page 42 of \cite{arora2009computational} for a definition of Karp-reducibility.) In particular, if any \textbf{NP}-hard problem is solvable in polynomial time, then \textbf{P} $=$ \textbf{NP}.

Now instead of restricting our attention to decision problems, we consider the computational complexity of more general problems. Recall that a \emph{problem} $P$ is just a function $P \colon \{ 0 ,1 \}^\ast \rightarrow \{ 0 ,1 \}^\ast$. We say that $P$ is in \textbf{FNP} if there is a deterministic polynomial time verifier that, given an arbitrary input pair $(x, y)$ where 
$x,y \in \{ 0 ,1 \}^\ast$, determines whether $P(x) = y$.




\subsection{Thurston norm}
\label{Sec:ThurstonNorm}

Let $M$ be any compact orientable 3-manifold. Thurston \cite{thurston1986norm} defined a semi-norm on the second homology group $H_2(M , \partial M ; \mathbb{R})$. This norm generalises the notion of knot genus, and for any homology class measures the minimum `complexity' between all properly embedded oriented surfaces representing that homology class. More precisely, for any integral homology class $a \in H_2(M , \partial M ; \mathbb{R})$ define the \emph{Thurston norm} of $a$, $x(a)$, as 
\[ x(a) = \min \{ \chi_-(S) \hspace{1mm} | \hspace{1mm} [S]=a,  \hspace{3mm} S \text{ is compact, oriented and properly embedded} \}. \]
This defines the norm for integral homology classes. One can extend this linearly to rational homology classes, and then extend it continuously to all real homology classes.

Consider the special case that $K$ is a knot of genus $g$ in $S^3$, and $M: = S^3 - N^\circ(K)$, where $N(K)$ is a tubular neighbourhood of $K$. The second homology group $H_2(M , \partial M; \mathbb{R})$ is isomorphic to $\mathbb{R}$ and the Thurston norm of a generator for the integral lattice
\[ H_2(M , \partial M ; \mathbb{Z}) \subset H_2(M , \partial M; \mathbb{R}) \]
is equal to $2g-1$ if $g \geq 1$, and $0$ otherwise. 

In general this might be a semi-norm as opposed to a norm, since one might be able to represent some non-trivial homology classes by a collection of spheres, discs, tori or annuli. However, if $W$ denotes the subspace of $H_2(M , \partial M ; \mathbb{R})$ with trivial Thurston norm, then one gets an induced norm on the quotient vector space $H_2(M , \partial M ; \mathbb{R})/W$. 

Thurston proved that the unit ball of this norm is a convex polyhedron. Given any norm on a vector space $V$, there is a corresponding dual norm on the dual vector space, that is the space of functionals on $V$. In our case, the dual space to $H_2(M , \partial M ; \mathbb{R})$ is $H^2(M , \partial M ; \mathbb{R})$. Thurston showed that the unit ball for the corresponding dual norm $x^\ast$ is a convex polyhedron with \emph{integral} vertices. For a thorough exposition of Thurston norm and examples see \cite{thurston1986norm, candel2003foliations}.

Finally, it is possible to define a norm $x_s$ using singular surfaces and allowing real coefficients. Thus, if $S_1 , \cdots, S_k$ are oriented singular surfaces in a 3-manifold $M$, and if $S = \sum a_i S_i$ is a real linear combination, representing a homology class $a$, we may define
\[ \chi_-(S) = \sum |a_i| \chi_-(S_i). \]
The singular norm $x_s$ is defined as 
\[ x_s(a) = \inf \{ \chi_-(S) \hspace{1mm} | \hspace{1mm} [S]=a \}. \]  


Gabai \cite{gabai1983foliations} proved the equivalence of the two norms $x$ and $x_s$, previously conjectured by Thurston \cite{thurston1986norm}.

\begin{thm}[Gabai]
Let $M$ be a compact oriented 3-manifold. Then on $H_2(M)$ or $H_2(M, \partial M)$, $x_s = x$ where $x_s$ denotes the norm on homology based on singular surfaces.  
\label{singular-norm}
\end{thm}

\subsection{Bareiss algorithm for solving linear equations} 
\label{SubSec:Bareiss}

Gaussian elimination is a useful method for solving a system of linear equations with integral coefficients, computing determinants and calculating their echelon form. The algorithm uses $O(n^3)$ arithmetic operations, where $n$ is the maximum of the number of variables and the number of equations. One caveat is that the intermediate values for the entries during the process can get large. An algorithm due to Bareiss resolves this issue. If the maximum number of bits for entries of the input is $L$, then the running time of the algorithm is at most a polynomial function of $n+L$. Moreover, no intermediate value (including the final answer) needs more than $O(n \log(n) + n L)$ bits \cite{bareiss1968sylvester}.

\subsection{Mixed integer programming}
\label{SubSec:MixedIntegerProg}

This refers to the following decision problem.  Let $n \geq 0$ and $m > 0$ be integers, and let $k$  be a positive integer satisfying $k \geq n$. Let $A$ be an $m \times k$ matrix with integer coefficients, and let $b \in \mathbb{Z}^m$. Then the problem asks whether there is an $x = (x_1, \dots, x_k)^T \in \mathbb{R}^k$ such that
\[ Ax \leq b\]
\[ x_i \in \mathbb{Z} \textrm{ for all } i \textrm{ satisfying } 1 \leq i \leq n.\]
The size of the input is given by $k + m + C_{\mathrm{dig}}(A) + C_{\mathrm{dig}}(b)$.
Lenstra \cite{lenstra1983integer} provided an algorithm to solve this problem that runs in polynomial time for any fixed value of $n$. 

It is also shown in \cite{lenstra1983integer}, using estimates of von zur Gathen and Sieveking \cite{VzGSieveking}, that if the above instance of Mixed Integer Programming does have a positive solution $x$, then it has one for which $C_{\mathrm{dig}}(x)$ is bounded above by a polynomial function of the size of the input.

Figure \ref{linear programming} shows an example of Mixed Integer Programming where
\begin{enumerate}
	\item the shaded region is the \emph{feasible region} namely $A x \leq b$ where $x \in \mathbb{R}^2$;
	\item the dots indicate integral points inside the feasible region.
	\end{enumerate}
In this example, there is at least one integral point inside the feasible region and hence the answer is yes.

\begin{figure}	
	\centering
	\includegraphics[width= 2 in]{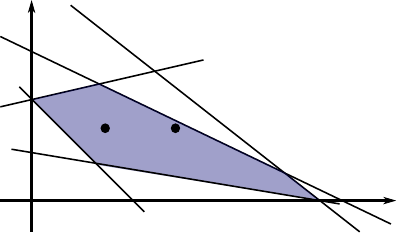}
	\caption{Integer Linear Programming}
	\label{linear programming}
\end{figure}

\subsection{Polyhedra and their duals} 

Our exposition is from \cite{brondsted2012introduction} and we refer the reader to that for more details and proofs. A set of points $\{ y_0 , y_1 , \cdots, y_m \} \subset \mathbb{R}^d$ is \emph{affinely independent} if the vectors $y_1 -y_0 , \cdots, y_m- y_0$ are linearly independent. A \emph{polytope} $P$ is the convex hull of a non-empty finite set $\{ x_1 , \cdots, x_n \}$ in $\mathbb{R}^d$. We say $P$ is $k$-dimensional if some $(k+1)$-subfamily of $\{ x_1 , \cdots, x_n \}$ is affinely independent, and $k$ is maximal with respect to this property. A convex subset $F$ of $P$ is called a \emph{face} of $P$ if for any two distinct points $y,z \in P$ such that $]y,z[ \hspace{1mm} \cap F$ is non-empty, we have $[y,z] \subset F$. Here $]y,z[$ and $[y,z]$ denote the open and closed segments connecting $y$ and $z$ respectively. A face $F$ is \emph{proper} if $F \neq \emptyset, P$. A point $x \in P$ is a \emph{vertex} if $\{ x \}$ is a face. A \emph{facet} $F$ of $P$ is a face of $P$ with $\dim(F) = \dim(P)-1$. Every face $F$ of $P$ is itself a polytope, and coincides with the convex hull of the set of vertices of $P$ that lie in $F$. Every proper face $F$ of $P$ is the intersection of facets of $P$ containing $F$ (see Theorem 10.4 of \cite{brondsted2012introduction}).

The intersection of any family of faces of $P$ is again a face. For any family $\mathcal{A}$ of faces of $P$, there is a largest face contained in all members of $\mathcal{A}$ denoted by $\inf \mathcal{A}$, and there is a smallest face that contains all members of $\mathcal{A}$ denoted by $\sup \mathcal{A}$. Denote the set of faces of $P$ by $\mathcal{F}(P)$ and let $\subset$ denote inclusion. Therefore, the partially ordered set $(\mathcal{F}(P), \subset)$ is a \emph{complete lattice} with the lattice operations $\inf \mathcal{A}$ and $\sup \mathcal{A}$. The pair $(\mathcal{F}(P), \subset)$ is called the \emph{face-lattice} of $P$.

Let $P$ be a $d$-dimensional polytope in $\mathbb{R}^d$ containing the origin. Define the \emph{dual} of $P$ as 
\[ P^\ast:= \{ y \in \mathbb{R}^d \hspace{1mm}|\hspace{1mm} \sup_{x \in P} \langle x, y \rangle \leq 1 \}. \]
For any face $F$ of $P$, define the dual face $F^{\triangle}$ as
\[ F^\triangle:= \{ y \in P^\ast \hspace{1mm}|\hspace{1mm} \sup_{x \in F} \langle x, y \rangle = 1 \}.\]
We have $(P^\ast)^\ast = P$, and $(F^\triangle)^\triangle = F$. There is a one-to-one correspondence between faces $F $ of $P$ and faces $F^\triangle$ of $P^\ast$, and 
\[ \dim(F) + \dim(F^\triangle) =d-1. \] 
Moreover, the mapping $F \mapsto F^\triangle$ defines an \emph{anti-isomorphism} of face-lattices (see Corollary 6.8 of \cite{brondsted2012introduction})
\[ (\mathcal{F}(\mathcal{P}), \subset) \rightarrow (\mathcal{F}(\mathcal{P}^*), \subset). \]

A subset $Q$ of $\mathbb{R}^d$ is called a \emph{polyhedral set} if $Q$ is the intersection of a finite number of closed half-spaces or $Q = \mathbb{R}^d$. Polytopes are precisely the non-empty bounded polyhedral sets.





\subsection{Pseudo-manifolds, orientability and degree of mappings }
At some point in this article, we need to talk about the degree of a mapping between two topological spaces that a-priori are not manifolds. They are similar to manifolds, but with particular types of singularities. The following discussion is from \cite{birman1980seifert}. `A \emph{closed pseudo-manifold} is defined as follows:

PM1) It is a \emph{pure}, finite $n$-dimensional simplicial complex  ($n \geq 1$); by pure we mean that each $k$-simplex is a face of at least one $n$-simplex (\emph{purity condition}).

PM2) Each $(n - 1)$-simplex is a face of exactly two $n$-simplices (\emph{non-branching condition}).

PM3) Every two $n$-simplexes can be connected by means of a series of alternating $n$- and $(n - 1)$-simplexes, each of which is incident with its successor (\emph{connectivity condition}).

A closed pseudo-manifold is said to be \emph{orientable} if each of its $n$-simplices can be oriented coherently, that is, oriented so that opposite orientations are induced in each $(n - 1)$-simplex by the two adjoining $n$-simplices. 

A closed $n$-chain on an orientable and coherently oriented closed pseudo-manifold is completely determined whenever one knows how often a single, arbitrarily chosen, oriented $n$-simplex appears in the chain. This is so 
because each of the $n$-simplices adjoining this simplex must appear equally often, from (PM3). One can reach each $n$-simplex by moving successively through adjoining simplices; hence all $n$-simplices must appear equally often. Consequently  the  $n$-th homology group is the free cyclic group. In other words, the $n$-th Betti number is equal to 1. A basis for this group is one of the two chains which arise by virtue of the coherent orientation of the pseudo-manifold.' \\

A choice for one of these chains is an \emph{orientation} of the pseudo-manifold and its homology class is then called the \emph{fundamental class}.

Let $D$ and $R$ be oriented pseudo-manifolds of dimension $n$, and $f \colon D \longrightarrow R$ be a continuous map. The $n$-th homology groups of $D$ and $R$ are both isomorphic to a free cyclic group. Denote by $[D]$ and $[R]$ the fundamental homology classes of $D$ and $R$ respectively. Then there is an integral number $d$ such that $f_\ast([D])$ is homologous to $d [R]$. This number $d$ is defined as the \emph{degree} of $f$. Similar to the case of manifolds, one can compute the degree by counting signed preimages of a generic point, where the sign depends on whether the map is locally orientation-preserving or not.

\subsection{Normal surfaces}
\label{Sec:Normal}
The theory of normal surfaces was introduced by Kneser in \cite{kneser1929geschlossene} where he proved a prime decomposition theorem for compact 3-manifolds, and was extended by Haken in his work on algorithmic recognition of the unknot \cite{haken1961theorie}. Let $\mathcal{T}$ be a triangulation of a compact 3-manifold $M$. A surface $S$ properly embedded in $M$ is said to be \emph{normal} if it intersects each tetrahedron in a collection of disjoint triangles and squares, as shown in Figure \ref{Fig:Normal}.

\begin{figure}[h]	
	\centering
	\includegraphics[width= 3 in]{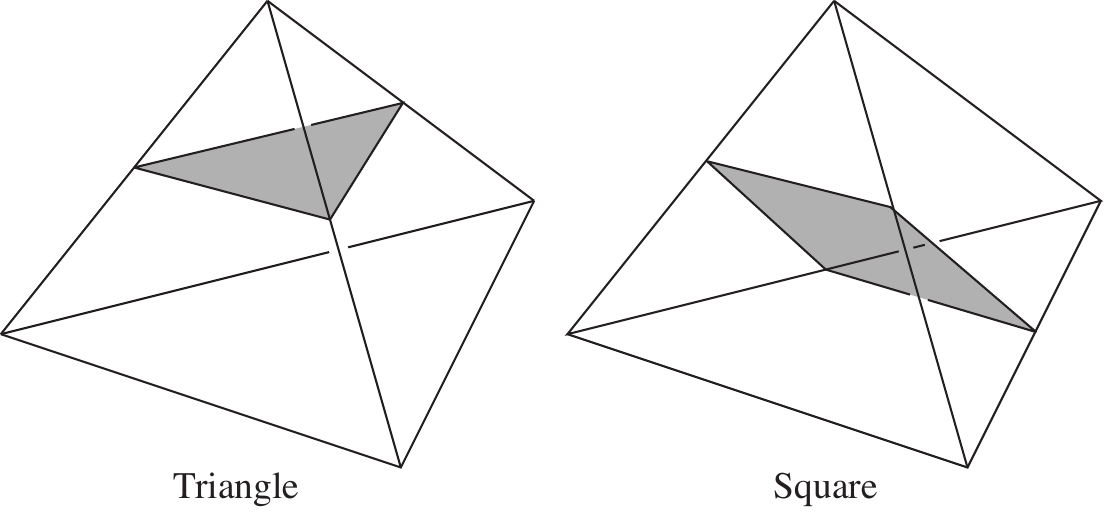}
	\caption{A triangle and a square}
	\label{Fig:Normal}
\end{figure}

In each tetrahedron, there are 4 types of triangles and 3 types of squares. Thus, in total, there are $7t$ types of triangles and squares in $\mathcal{T}$, where $t$ is the number of tetrahedra in $\mathcal{T}$. A normal surface $S$ determines a list of $7t$ non-negative integers, which count the number of triangles and squares of each type in $S$. This list is called the \emph{vector} for $S$ and is denoted by $(S)$.

The normal surface $S$ is said to be \emph{fundamental} if $(S)$ cannot be written as $(S_1) + (S_2)$ for non-empty properly embedded normal surfaces $S_1$ and $S_2$. 
It is said to be a \emph{vertex surface} if no non-zero multiple of $(S)$ can be written as $(S_1) + (S_2)$ for non-empty properly embedded normal surfaces $S_1$ and $S_2$. This has an alternative interpretation in terms of the normal solution space, as follows.

The \emph{normal solution space} $\mathcal{N}(\mathcal{T})$ is a subset of $\mathbb{R}^{7t}$. The co-ordinates of $\mathbb{R}^{7t}$ correspond to the $7t$ types of triangles and squares in $\mathcal{T}$. The subset $\mathcal{N}(\mathcal{T})$ consists of those points in $\mathbb{R}^{7t}$ where every co-ordinate is non-negative and that satisfy the normal \emph{matching equations} and \emph{compatibility conditions}. There is one matching equation for each type of normal arc in each face of $\mathcal{T}$ not lying in $\partial M$. The equation asserts that in each of the two tetrahedra adjacent to that face, the total number of triangles and squares that intersect the given face in the given arc type are equal. The compatibility conditions assert that for different types of squares within a tetrahedron, at least one of the corresponding co-ordinates is zero. For any properly embedded normal surface $S$, its vector $(S)$ lies in $\mathcal{N}(\mathcal{T})$. Indeed, the set of points in $\mathcal{N}(\mathcal{T})$ that are a vector of a properly embedded normal surface is precisely $\mathcal{N}(\mathcal{T}) \cap \mathbb{Z}^{7t}$.

The \emph{projective solution space} $\mathcal{P}(\mathcal{T})$ is the intersection 
$$\mathcal{N}(\mathcal{T}) \cap\{ (x_1, \cdots, x_{7t}) : x_1 + \cdots + x_{7t} = 1 \}.$$
It is shown in \cite{matveev2007algorithmic} that $\mathcal{P}(\mathcal{T})$ is a union of convex polyhedra. A normal surface $S$ is \emph{carried} by a face of $\mathcal{P}(\mathcal{T})$ if its vector $(S)$ lies on a ray through the origin of $\mathbb{R}^{7t}$ that goes through that face. When a normal surface $S$ is carried by a face $C$, and $(S) = (S_1) + (S_2)$ for normal surfaces $S_1$ and $S_2$, then $S_1$ and $S_2$ are also carried by $C$. The reason for this is that $C$ is the intersection between $\mathcal{P}(\mathcal{T})$ and some hyperplanes of the form $\{ x_i = 0 \}$. Since $(S) = (S_1) + (S_2)$, then $(S_1)$ and $(S_2)$ also lie in these hyperplanes and hence also are carried by $C$.

A normal surface $S$ is a vertex surface exactly when some non-zero multiple of $(S)$ is a vertex of one of the polyhedra of $\mathcal{P}(\mathcal{T})$. Using this observation, it was shown by Hass and Lagarias (Lemma 3.2 in \cite{HassLagarias}) that each co-ordinate of the vector of a vertex normal surface is at most $2^{7t-1}$. Hence, the number of points of intersection between a vertex normal surface and the 1-skeleton of $\mathcal{T}$ is at most $28t2^{7t-1}$. They also showed that each co-ordinate of a fundamental normal surface in $\mathcal{T}$ has modulus at most $t2^{7t+2}$.

A common measure of complexity for a normal surface $S$ is its \emph{weight} $w(S)$ which is its number of intersections with the 1-skeleton of $\mathcal{T}$.

\section{Main Theorems}

In this section, we prove Theorems \ref{thm:Thurston complexity for a knot} and \ref{main:boundary}, assuming various ingredients that will be proved in later sections. 


\theoremstyle{theorem}
\newtheorem*{Thurston complexity for a knot}{Theorem \ref{thm:Thurston complexity for a knot}}
\begin{Thurston complexity for a knot}
		Let $M$ be a compact, orientable 3-manifold given by a fixed diagram $D$ for $\Gamma \cup L$, where $M$ is obtained from $S^3$ by removing an open regular neighbourhood of the graph $\Gamma$ and performing surgery on the framed link $L$. The problem \textsc{Thurston complexity for a knot in the fixed 3-manifold $M$} lies in \textbf{co-NP}.
\end{Thurston complexity for a knot}	

\begin{proof} 

We are given a diagram for $K \cup \Gamma \cup L$, which contains $D$ as a sub-diagram, where $K$ is our given knot. Let $c$ be the total crossing number of this diagram of $K$. Recall that this is the number of crossings between $K$ and itself, and between $K$ and $\Gamma \cup L$. Set $X:= M - N^{\circ}(K)$ to be the exterior of a tubular neighbourhood of $K$ in $M$.\\
 
\textbf{Step 1}:  By Theorem \ref{basis}, we can construct a triangulation $\mathcal{T}$ of $X = M - N^\circ(K)$, and simplicial 1-cocycles $\phi_1 , \cdots , \phi_b$ such that
\begin{enumerate} 
\item the number of tetrahedra $|\mathcal{T}|$ of $\mathcal{T}$ is at most a linear function of $c$;
\item the cocycles $\phi_1, \cdots , \phi_b$ form an integral basis for $H^1(X ; \mathbb{Z})$;
\item the \emph{unary} complexity $\sum_i C_{\mathrm{una}}(\phi_i)$ is at most a polynomial function of $c$; 
\item $\mathcal{T}$ can be extended to a triangulation $\mathcal{T}'$ of $M$ with $O(c)$ tetrahedra, in which $K$ is simplicial.
\end{enumerate}
Moreover, the construction of $\mathcal{T}$, $\mathcal{T}'$, and $\phi_1, \cdots, \phi_b$ can be done in polynomial time in $c$. \\

\textbf{Step 2:} We check whether $K$ is homologically trivial in $M$, as otherwise there is no Seifert surface for $K$ and the genus of $K$ is $\infty$. We do this by considering the triangulation $\mathcal{T}'$ of $M$ in which $K$ is simplicial and then determining whether $K$ is the boundary of a simplicial 2-chain. This can be done in time that is polynomial in $c$, using the Bareiss algorithm for solving linear equations.

Since $K$ is homologically trivial, it has a longitude denoted by $\ell$. Recall that this is the boundary of a Seifert surface $S$ for $K$. The longitude is unique up to sign, for the following reason. If $\ell'$ is any other longitude, the intersection number $[\ell'].[\ell]$ on $\partial X$ equals the intersection  number $[\ell'].[S]$ in $X$, but this is zero because $\ell'$ is homologically trivial in $X$. Again using the Bareiss algorithm, the longitude $\ell$ on $\partial N(K)$ can be determined in time that is polynomial in $c$ and, when it is represented as a simplicial 1-cycle, its complexity $C_{\mathrm dig}(\ell)$ is at most a polynomial function of c.
\\

\textbf{Step 3}: Note the first Betti number of $X$ is bounded above by the constant $B = b_1(M)+1$, since we are drilling a knot from $M$. Therefore, we can use Theorem \ref{thurston ball} to compute the unit ball of the Thurston norm on $H^1(X ; \mathbb{R})$, using a non-deterministic Turing machine. Here $H^1(X ; \mathbb{R})$ has been identified with $H_2(X , \partial X ; \mathbb{R})$ using Poincar\'{e} duality. 

Note that the size of the input, that is the sum of the number of tetrahedra and $\sum_i C_{\mathrm{una}}(\phi_i)$, is at most a polynomial function of $c$. Therefore, this can be done in time that is at most polynomially large in $c$. Hence, we can construct the following:  
\begin{enumerate}
\item  A basis $\{ w_1 , \cdots , w_r \}$ for the subspace $W$ of $H^1(X ; \mathbb{R})$ with trivial Thurston norm. Each $w_i$ is an integral cocycle and is written as a linear combination of the given cocycles $\{ \phi_1 , \cdots , \phi_b \}$. Denote by $p$ the projection map from $H^1(X ; \mathbb{R})$ to $H^1(X ; \mathbb{R})/W$.

\item  A set of points $V \subset H^1(X ; \mathbb{Q})$ such that $p(V)$ is the set of vertices of the unit ball of $H^1(X ; \mathbb{R})/W$, together with a list $\mathcal{F}$ of subsets of $V$. Each element of $V$ is written in terms of the basis $\{ \phi_1 , \cdots , \phi_b \}$. We think of $\mathcal{F}$ as the list of faces of the unit ball for $H^1(X ; \mathbb{R})/W$, in the sense that for $F \in \mathcal{F}$ the set
\[ \{ p(v) \hspace{2mm} | \hspace{2mm} v \in F \}, \]
forms the set of vertices of some face of the unit ball for $H^1(X ; \mathbb{R})/W$. Moreover, this covers all the faces as we go over all the elements $F$ of $\mathcal{F}$.
\end{enumerate}
Because the problem \textsc{Thurston norm ball for 3-manifolds with $b_1 \leq B$} lies in \textbf{FNP}, the number of digits of the output is at most a polynomial function of the complexity of the input. Hence, $\sum_i C_{\mathrm{dig}}(w_i)$, $\sum_{v \in V} C_{\mathrm{dig}}(v)$ and $|\mathcal{F}|$ are all bounded above by polynomial functions of $c$.\\


\textbf{Step 4}: 
There is an identification between $H^1(X; \mathbb{Z})$ and $H_2(X, \partial X ; \mathbb{Z})$ using Poincar\'{e} duality, and there is a boundary map 
\[ H_2(X , \partial X; \mathbb{Z}) \longrightarrow H_1(\partial X ; \mathbb{Z}). \]
This in turn induces a boundary map 
\[ \partial \colon H^1(X ; \mathbb{Z}) \longrightarrow H_1(\partial X ; \mathbb{Z}). \]

For each facet $F = \{ u_1, \cdots, u_s \} $ of the unit ball of the Thurston norm on $H^1(X ; \mathbb{R} ) / W$, denote by $\text{Cone}(F)$ the cone over the face $F$:
\[ \text{Cone}(F) : = \{ r_1 u_1 + \cdots + r_s u_s  \hspace{1mm}| \hspace{1mm} r_1, \cdots, r_s \in \mathbb{R}_{\geq 0} \}. \]
Denote the Thurston semi-norm by $x \colon H^1(X ; \mathbb{R}) \rightarrow \mathbb{R}$.  For each facet $F$, consider the following minimum:
\[ m_F: = \min \hspace{1mm} \{ x(h) \hspace{2mm} | \hspace{2mm} h \in (W + \text{Cone}(F)) \cap H^1(X; \mathbb{Z}) \hspace{2mm} \text{and} \hspace{2mm}   \partial (h) = \pm [\ell] \} , \]
where $m_F$ is defined to be $\infty$ if there is no integral homology class in $W + \text{Cone}(F)$ and with boundary $\pm [\ell]$. This integer $m_F$ is given to us non-deterministically. We will show below that the number of digits of $m_F$ is bounded above by a polynomial function of $c$. We verify that $m_F$ is indeed the above minimum, in time that is at most polynomially large in $c$, by the following argument.

Let $h$ be any element of $W + \text{Cone}(F)$. We may write
\begin{equation}
h = \beta_1 w_1 + \cdots + \beta_r w_r +  \alpha_1 u_1 + \cdots + \alpha_s u_s,
\end{equation} 
for some $\beta_1, \cdots, \beta_r \in \mathbb{R}$ and $\alpha_1, \cdots, \alpha_s \in \mathbb{R}_{\geq 0}$. 
We also require that $h$ lies in $H^1(X; \mathbb{Z})$ which is the condition
\begin{equation}
h = \gamma_1 \phi_1 + \cdots + \gamma_b \phi_b,
\end{equation}
for some $\gamma_1, \cdots, \gamma_b \in \mathbb{Z}$. The condition $\partial (h) = \pm [\ell]$ translates into 
\begin{equation}
\beta_1 \partial(w_1) + \cdots + \beta_r \partial (w_r) + \alpha_1 \partial(u_1)+ \cdots + \alpha_s \partial(u_s) = \pm [\ell].
\end{equation}
Since each of the vertices $\{ u_1 , \cdots , u_s \}$ has Thurston norm equal to one and they all lie on the same face, we have 
\[ x(h) = \alpha_1 + \cdots + \alpha_s. \]
Therefore, we are looking to find the minimum $m_F$ of the linear functional $\alpha_1 + \cdots + \alpha_s$ under the conditions 

i) $\alpha_i \in \mathbb{R}_{\geq 0}$, $\beta_i \in \mathbb{R}$, $\gamma_i \in \mathbb{Z}$;

ii) $(1)=(2)$ (by which we mean putting the right hand sides of the equations equal) and $(3)$.

As explained in Section \ref{SubSec:MixedIntegerProg}, since the number $b$ of integer variables is fixed, for each facet we can verify whether the above equations have a solution, in time that is bounded above by a polynomial function of the number of real variables, the number of equations and the number of bits encoding the coefficients of the linear constraints. Hence this can be done in time that is a polynomial function of $c$. Moreover, for each facet that the above equations have a solution, there is a solution where $\sum_i C_{\mathrm{dig}}(\alpha_i)$ is bounded above by a polynomial function of $c$. Hence, the number of digits of $m_F$ is bounded from above by a polynomial function of $c$.

If $m_F =  \infty$ then we verify using Mixed Integer Programming that the above equations have no solution.  If $m_F$ is finite, we verify that the condition
\[ x(h) = m_F \]
is satisfied for some $h$, whereas the
condition 
\[x(h) \leq m_F - 1\]
has no solution. (Note that the Thurston norm takes only integer values on elements of $H^1(X; \mathbb{Z})$ and so we do not need to consider the possibility that $x(h)$ might lie strictly between $m_F-1$ and $m_F$.) These are again instances of Mixed Integer Programming, and can be computed in polynomial time in $c$. See Figure \ref{knot genus} for an example, with the following properties.
\begin{enumerate}
	\item The octahedron is the unit ball of the Thurston norm ($W = \{0\}$). 
	\item The affine plane $P$ is the location of points with boundary equal to $[\ell]$. In this example, $P$ is disjoint from the unit ball.
	\item The shaded region on $P$ is the intersection of $P$ with the cone over the shaded face $F$ of the unit ball. Here the shaded face $F$ is a triangle, and its projection to $P$ is a degenerate (non-compact) triangle since one edge of $F$ happened to be parallel to $P$. 
	\item The dots on $P$ indicate the integral points on $P$ lying in the cone over $F$. 
	\item The face $F$ determines the equation for the linear functional $x(h)$. The Mixed Integer Programming problem asks whether there is an integral point $h$ on $P$ that satisfies $x(h) \leq m_F$ (respectively $m_F -1$) and the constraints i) and ii) above. 
\end{enumerate}

\begin{figure}
	\labellist
	\pinlabel $P$ at 200 50
	\pinlabel $F$ at 100 100 
	\endlabellist
	
	\centering
	\includegraphics[width=3 in]{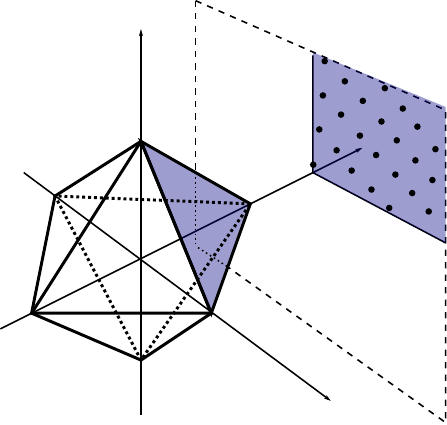}
	\caption{Finding the minimum Thurston complexity between Seifert surfaces coming from a single face of the Thurston norm ball}
	\label{knot genus}
\end{figure} 

\textbf{Step 5}: Since $K$ is homologically trivial, $m_F$ is finite for at least one facet $F$. The minimum of $m_F$ over all facets $F$ of the unit ball of the Thurston norm for $H^1(X ; \mathbb{R})/W$ is equal to the Thurston complexity $c_\text{Th}(K)$ of $K$. Therefore, we can check if this minimum is greater than the given integer $g$ or not. Note that the number of facets is at most polynomially large in $c$ by the combination of Theorems \ref{basis} and \ref{number of faces}. Hence, the algorithm runs in time that is at most polynomially large in $c$.\\ 

This finishes the non-deterministic algorithm, thereby establishing that \textsc{Thurston complexity for a knot in the fixed 3-manifold $M$} lies in \textbf{co-NP}. 
\end{proof}

The next lemma translates the problem of computing the genus of a knot $K$ to the computation of the Thurston complexity for $K$. This is where the toroidal boundary hypothesis will be used in the proof of Theorem \ref{main:boundary}.

\begin{lem}
	Let $M$ be a compact orientable 3-manifold with toroidal boundary, and $K$ be a homologically trivial oriented knot in $M$. Let $c_\mathrm{Th}(K)$ be the Thurston complexity for $K$. The genus of $K$ is equal to $\frac{1}{2} (c_\mathrm{Th}(K)+1)$ if $c_\mathrm{Th}(K) \geq 1$, and is equal to $0$ otherwise. 
	\label{computing genus from Thurston norm}
\end{lem}

\begin{proof}
	Set $X := M - N^\circ(K)$ as the complement of a tubular neighbourhood of $K$ in $M$. Let
	\[ \partial \colon H_2(X, \partial X; \mathbb{Z}) \rightarrow H_1(\partial X; \mathbb{Z})\] 
	be the boundary map. Denote by $\ell $ the longitude of $K$. Recall that 
	\begin{align*}
	c_\mathrm{Th}(K): = \min \{ \chi_-(S) \hspace{1mm} | \hspace{1mm} & S \text{ is a compact oriented properly embedded surface in }X \\ & \text{ with } [\partial S] = [\ell] \in H_1(\partial X ; \mathbb{Z})   \}.
\end{align*}
	Note that a Seifert surface for $K$ has only one boundary component by definition, which lies on $\partial N(K)$, whereas a compact oriented surface whose boundary is in the same homology class as $[\ell] \in H_1(\partial X; \mathbb{Z})$ can have extra boundary components lying on $\partial M$. Hence, the lemma follows if we show that $c_\mathrm{Th}(K) = c_\mathrm{g}(K)$, where
	\begin{align*}
		c_\mathrm{g}(K): = \min \{ \chi_-(S) \hspace{1mm} | \hspace{1mm} & S \text{ is a compact oriented properly embedded surface in } X \\ & \text{ with } \partial S = \ell \text{ up to isotopy}   \}
	\end{align*}
	
	The inequality $c_\mathrm{Th}(K) \leq c_\mathrm{g}(K)$ is immediate, so we need to show that $c_\mathrm{g}(K) \leq c_\mathrm{Th}(K)$. 
	
	Let $S$ be a compact oriented surface with $[\partial S] = [\ell] \in H_1(\partial X ; \mathbb{Z})$. We will modify $S$ to construct a Seifert surface $S'$ with $\chi_-(S') \leq \chi_-(S)$. If any component of $\partial S$ bounds a disc in $\partial X$, we may pick one that is innermost in $\partial X$ and then attach the disc that it bounds to $S$. This does not increase $\chi_-(S)$. So, we may assume that no component of $\partial S$ bounds a disc in $\partial X$.
	For each torus component $T$ of $\partial M$, the union of components of $\partial S$ lying on $T$ is homologically trivial. Hence this union consists of parallel essential simple closed curves and there are two adjacent simple closed curves along $T$ that have opposite orientations; these two boundary components of $S$ can be tubed together to obtain a surface $F$ with the same Euler characteristic as $S$ but with fewer boundary components on $\partial M$. We claim that $\chi_-(F) \leq \chi_-(S)$. To see this, we analyse components of $F$ with positive Euler characteristic. Any component of $F$ with positive Euler characteristic is either 
	\begin{enumerate}
		\item an untouched component of $S$, or
		\item a disc obtained by tubing together a disc and an annulus component of $S$, or
		\item a sphere obtained by tubing together two disc components of $S$.
	\end{enumerate}
	This together with $\chi(F)=\chi(S)$ shows that $\chi_-(F) \leq \chi_-(S)$. Repeating this procedure by tubing boundary components together, we obtain a surface $F''$ disjoint from $\partial M$ such that $\chi_-(F'') \leq \chi_-(S)$. Now $\partial F''$ is homologous to $\ell$, and hence it consists of a union of curves parallel to $\ell$ on $\partial N(K)$. By repeating the tubing procedure as above, we can construct a Seifert surface $S'$ for $K$ with $\chi_-(S') \leq \chi_-(S)$. This shows that $c_\mathrm{g}(K) \leq c_\mathrm{Th}(K)$, which together with the trivial inequality $c_\mathrm{Th}(K) \leq c_\mathrm{g}(K)$ completes the proof.
\end{proof}

\theoremstyle{theorem}
\newtheorem*{main}{Theorem \ref{main:boundary}}
\begin{main}
	Let $M$ be a compact, orientable 3-manifold with toroidal boundary given by a fixed diagram $D$ for $\Gamma \cup L$, where $M$ is obtained from $S^3$ by removing an open regular neighbourhood of the link $\Gamma$ and performing surgery on the framed link $L$. The problem \textsc{Knot genus in the fixed 3-manifold $M$} lies in \textbf{co-NP}.
\end{main}

\begin{proof}
	We would like to verify in non-deterministic polynomial time whether the genus $g(K)$ of $K$ is greater than a given integer $g$ or not. Let $k = 2g-1$ if $g\geq 1$, and $k=0$ otherwise. Since $M$ has toroidal boundary, Lemma \ref{computing genus from Thurston norm} implies that $c_\mathrm{Th}(K)= 2 g(K) -1$ if $g(K) \geq 1$ and $c_\mathrm{Th}(K)=0$ otherwise. By Theorem \ref{thm:Thurston complexity for a knot}, we can verify in non-deterministic polynomial time whether the Thurston complexity $c_\mathrm{Th}(K)$ for $K$ is greater than $k$ or not. Equivalently, we can verify in non-deterministic polynomial time whether the genus $g(K)$ is greater than $g$ or not, completing the proof.
\end{proof}

\section{The number of faces and vertices of the Thurston norm ball }

In this section, we prove Theorem \ref{number of faces}, which provides an upper bound on the number of faces and vertices of the Thurston norm ball. The key to this is the following result, which controls the number of integral points in the dual norm ball.

\begin{thm}
Let $X$ be a compact orientable 3-manifold, and let $m$ be a natural number. Assume that there exist properly immersed oriented surfaces $S_1 , \cdots, S_b$ in $X$ such that their homology classes form a basis for $H_2(X, \partial X; \mathbb{R}) $, and for each $1 \leq i \leq b$ we have $|\chi_-(S_i)| \leq m$. Define $\mathcal{A}$ as the set of integral points inside $H^2(X, \partial X; \mathbb{Z}) \otimes \mathbb{Q}$ whose dual norm is at most one. The size of $\mathcal{A}$ is at most $(2m+1)^b$, where $b = b_1(X)$ is the first Betti number of $X$. 
\label{lattice points}
\end{thm}

\begin{proof}
Let $\langle \cdot , \cdot \rangle$ be the pairing between cohomology and homology. Define dual elements $e^1 , \cdots , e^b \in H^2(X, \partial X; \mathbb{Z}) \otimes \mathbb{Q}$ as
\[ \langle e^i , [S_j] \rangle = \delta_{ij}, \]
where $1 \leq i, j \leq b$, and $\delta_{ij}$ is the Kronecker function. Every integral point $u \in H^2(X, \partial X; \mathbb{Z}) \otimes \mathbb{Q}$ can be written as 
\[ u = \alpha_1 e^1 + \cdots + \alpha_b e^b, \]
where $\alpha_i$ are integers. This is because $u$ being integral means that its evaluation against each element of $H_2(X, \partial X; \mathbb{Z})$ is an integer. In particular,
$\alpha_i = \langle u, [S_i] \rangle$ is an integer. Assume that the dual norm of $u$ is at most one. By definition of the dual norm, for each $1 \leq i \leq b$ we have:
\[ |\langle u , [S_i] \rangle| \leq x([S_i]) = x_s([S_i]), \]
where $x([S_i])$ and $x_s([S_i])$ are the Thurston norm and the singular Thurston norm of $[S_i]$, and the last equality is by Theorem \ref{singular-norm}. Since $|\chi_-(S_i)| \leq m$, we have 
\[ x_s([S_i]) \leq m. \]
Combining the two inequalities implies that
\[ |\alpha_i| = |\langle u , [S_i] \rangle| \leq x([S_i]) = x_s([S_i]) \leq m. \]
Since $-m \leq \alpha_i \leq m$ is an integer, there are at most $2m+1$ possibilities for each coordinate of the tuple $(\alpha_1, \cdots , \alpha_b)$. Therefore the number of possibilities for $u$ is at most $(2m+1)^b$.
\end{proof}

\theoremstyle{theorem}
\newtheorem*{number of faces}{Theorem \ref{number of faces}}
\begin{number of faces}
Let $X$ be a compact orientable 3-manifold, and let $m$ be a natural number. Assume that there exist properly immersed oriented surfaces $S_1 , \cdots, S_b$ in $X$ such that their homology classes form a basis for $H_2(X, \partial X; \mathbb{R}) $, and for each $1 \leq i \leq b$ we have $|\chi_-(S_i)| \leq m$. Denote by $W$ the subspace of $H^1(X  ; \mathbb{R})$ with trivial Thurston norm. 
The number of facets of the unit ball for the induced Thurston norm on $H^1(X  ; \mathbb{R})/W$ is at most $(2m+1)^b$, where $b = b_1(X)$ is the first Betti number of $X$. 
Hence, the number of vertices is at most $(2m+1)^{b^2}$ and the number of faces is at most $b(2m+1)^{b^2}$.
\end{number of faces}

\begin{proof}
Note we have identified $H_2(X, \partial X ; \mathbb{Z})$ with $H^1(X ; \mathbb{Z})$ using Poincar\'{e} duality. Facets of the unit ball for $H^1(X; \mathbb{R})/W$ correspond to the vertices of the dual ball. As the vertices of the dual ball are integral and have dual norm equal to one, the number of them is at most $(2m+1)^b$ by Theorem \ref{lattice points}. This proves the first part of the theorem.

Let $d$ be the dimension of $H^1(X; \mathbb{R})/W$; hence $d \leq b$. Every $k$-dimensional face of the unit ball is the intersection of $(d-k)$ facets. Hence, the number of $k$-dimensional faces is at most 
\[ {(2m+1)^b \choose d-k}.\]
As a result, the total number of faces is at most 
\[  { (2m+1)^b \choose 1}+ {(2m+1)^b \choose 2}+ \cdots + {(2m+1)^b \choose b}, \]
which is bounded above by $b (2m+1)^{b^2}$. In particular, the number of vertices is at most $(2m+1)^{b^2}$.
\end{proof}

\section{A basis for the homology of a knot complement with small Thurston complexity} 

Let $\Gamma$ be a finite graph piecewise linearly embedded in $S^3$ and let $L$ be a framed link in the complement of $\Gamma$. Let $M$ be the compact orientable 3-manifold given by removing an open regular neighbourhood of $\Gamma$ and performing surgery along $L$. We are considering a knot $K$ in $M$ given by a diagram for $K \cup \Gamma \cup L$. In this section, we show how to compute bases for $H^1(M)$ and $H^1(M - N^\circ(K))$. From these, we will be able to construct a basis for $H_2(M - N^\circ(K), \partial M \cup \partial N(K))$ with relatively small Thurston complexity.

Our first step is to construct bases for $H^1(S^3 - N^\circ(\Gamma \cup L))$ and $H^1(S^3 - N^\circ(K \cup \Gamma \cup L))$ using the following lemma.

\begin{lem}
\label{Lem:GraphComplement}
Let $G$ be a finite graph piecewise linearly embedded in $S^3$, possibly with multiple edges between vertices and edge loops. Then $H^1(S^3 - N^\circ(G))$ has the following basis. Pick a maximal forest $F$ in $G$. For each edge $e\in G - F$, orient it in some way and let $L_e$ be the knot that starts at the initial vertex of $e$, runs along $e$ and then back to the start of $e$ through an embedded path in $F$. Let $\psi_e$ be the homomorphism $\pi_1(S^3 - N^\circ(G)) \rightarrow \mathbb{Z}$ that sends a loop in $S^3 - N^\circ(G)$ to its linking number with $L_e$. Then $\{ \psi_e : e \textrm{ is an edge of } G - F \}$ forms an integral basis for $H^1(S^3 - N^\circ(G))$.
\end{lem}

\begin{proof}
Note first that $\psi_e$ really is a homomorphism $\pi_1(S^3 - N^\circ(G)) \rightarrow \mathbb{Z}$ since any homotopically trivial loop is sent to zero. These homomorphisms form linearly independent elements of $H^1(S^3 - N^\circ(G))$ because $\psi_e$ evaluates to $1$ on the meridian of $e$, but evaluates to $0$ on the meridian of any other edge of $G - F$. By Alexander duality, $b_1(S^3 - N^\circ(G)) = b_1(G)$, which is equal to the number of edges in $G - F$. So, $\{ \psi_e : e \textrm{ is an edge of } G - F \}$ forms a rational basis for $H^1(S^3 - N^\circ(G); \mathbb{Q})$. In fact, it forms an integral basis for $H^1(S^3 - N^\circ(G))$, for the following reason. Any element $\psi \in H^1(S^3 - N^\circ(G))$ is a linear combination $\sum_e \lambda_e \psi_e$, where each $\lambda_e \in \mathbb{Q}$. Since $\psi$ is integral, its evaluation on the meridian of an edge $e$ of $G - F$ is integral. But this number is $\lambda_e$.
\end{proof}

We now consider how to compute $H^1(M)$. We obtain $M$ from $S^3 - N^\circ(\Gamma \cup L)$ by attaching solid tori, each of which can be viewed as a 2-cell and a 3-cell. So,  $H^1(M)$ can be viewed as the subgroup of $H^1(S^3 - N^\circ(\Gamma \cup L))$ consisting of those classes that evaluate to zero on each of the surgery slopes
of the framed link $L$. This subgroup can be expressed in terms of the \emph{generalised linking matrix} of $\Gamma \cup L$.

Recall that the \emph{linking matrix} for the oriented framed link $L = L_1 \cup \cdots \cup L_{|L|}$ is defined as the $|L| \times |L|$ symmetric matrix whose $(i,j)$ entry is equal to $\ell k (L_i , L_j)$ when $i \neq j$, and is equal to the framing of $L_i$ when $i = j$. Here $\ell k (L_i , L_j)$ is the linking number of $L_i$ and $L_j$, where $L_i$ and $L_j$ are considered as disjoint knots in $S^3$. More generally, we define the \emph{generalised linking matrix} $A$ of $\Gamma \cup L$ to have rows given by the components of $L$ and columns given by the components of $L$ and also by the edges of $\Gamma - F$, where $F$ is a maximal forest in $\Gamma$. For a component $L_i$ of $L$ and an edge $e$ of $\Gamma - F$, the corresponding entry of $A$ is $\ell k (L_i , L_e)$, where $L_e$ is the knot defined as in Lemma \ref{Lem:GraphComplement}. Similarly, for each component $L_i$ of $L$ and each component $L_j$ of $L$ with $i \not= j$, the corresponding entry of $A$ is $\ell k (L_i , L_j)$. Finally, when $L_i = L_j$, the corresponding entry of $A$ is the framing of $L_i$. 

Let $k$ be the number of columns of the generalised linking matrix. Thus, $k$ is the sum of the number of components of $L$ and the number of edges of $\Gamma - F$. In other words, $k = b_1( \Gamma \cup L)$. To make the notation more uniform, we identify the edges in $\Gamma - F$ by numbers $|L|+1 \leq j \leq k$ and refer to $L_e$ for $e \in \Gamma - F$ by $L_j$.

\begin{lem}
The cohomology group $H^1(M)$ is isomorphic to the subgroup of $H^1( S^3 - N^\circ(\Gamma \cup L)) = \mathbb{Z}^k$ given by the kernel of the generalised linking matrix.
\end{lem}

\begin{proof} We have already identified $H^1( S^3 - N^\circ(\Gamma \cup L))$ by specifying an integral basis for it in Lemma \ref{Lem:GraphComplement}. The subgroup  $H^1(M)$ consists of those classes that evaluate to zero on each of the surgery slopes of the framed link $L$. When we write the surgery slope of $L_i$ as a linear combination of the basis elements, the coefficients are precisely the entries of the $i$th row of the generalised linking matrix.
\end{proof}

We now wish to compute $H^1(M - N^\circ(K))$. This can be viewed as containing $H^1(M)$ as a subgroup, using the long exact sequence of the pair $(M, M - N^\circ(K))$:
$$0 = H^1(M, M - N^\circ(K)) \rightarrow H^1(M) \rightarrow H^1(M - N^\circ(K)) \rightarrow H^2(M, M - N^\circ(K)) \rightarrow H^2(M).$$
Now, using excision and Poincar\'e duality,
$$H^2(M, M - N^\circ(K)) \cong H^2(N(K), \partial N(K)) \cong H_1(N(K))$$
and
$H^2(M) \cong H_1(M,\partial M)$. So, under our assumption that $K$ is homologically trivial, the map  $H^2(M, M - N^\circ(K)) \rightarrow H^2(M)$ is the trivial map. Therefore, $H^1(M - N^\circ(K))$ can be viewed as adding on a $\mathbb{Z}$ summand to $H^1(M)$. This summand is isomorphic to $H^2(M, M - N^\circ(K))$. 

We now wish to construct an integral basis for $H^1(M - N^\circ(K))$. Such a basis can be found by starting with an integral basis for $H^1(M)$ and taking its image in $H^1(M - N^\circ(K))$, and then adding one more element. This extra element must map to a generator for $H^2(M, M - N^\circ(K))$. We will build explicit cocycles representing this basis.

\theoremstyle{theorem}
\newtheorem*{basis}{Theorem \ref{basis}}
\begin{basis}
Let $M$ be a compact orientable 3-manifold given by removing an open regular neighbourhood of a (possibly empty) graph $\Gamma$ in $S^3$ and performing integral surgery on a framed link $L$ in the complement of $\Gamma$. Let $D$ be a fixed diagram for $\Gamma \cup L$ where the surgery slopes on $L$ coincide with the diagrammatic framing. Let $K$ be a homologically trivial knot in $M$, given by a diagram of $K \cup \Gamma \cup L$ that contains $D$ as a sub-diagram. Let $c$ be the total crossing number of $K$. Set $X = M - N^\circ(K)$ as the exterior of $K$ in $M $. There is an algorithm that builds a triangulation of $X$ with $O(c)$ tetrahedra, together with simplicial $1$-cocycles $\phi_1 , \cdots , \phi_b$ that form an integral basis for $H^1(X ; \mathbb{Z})$ with $\sum_i C_{\mathrm{una}}(\phi_i)$ at most $O(c^4)$. Moreover, the algorithm extends this triangulation of $X$ to a triangulation of $M$ with $O(c)$ tetrahedra, in which $K$ is simplicial. The algorithm runs in time polynomial in $c$. All the above implicit constants depend on the manifold $M$ and not the knot $K$.
\end{basis}

\begin{proof}

\textbf{Step 1:} Building a triangulation of $S^3 - N^\circ(K \cup \Gamma \cup L)$. \\

We view $S^3$ as the union of $\mathbb{R}^3$ and a point at infinity. We will arrange for $K \cup \Gamma \cup L$ to sit inside $\mathbb{R}^3$ as specified by the diagram. Thus, the vertical projection map $\mathbb{R}^3 \rightarrow \mathbb{R}^2$ onto the first two co-ordinates will project $K \cup \Gamma \cup L$ onto the underlying planar graph specified by the diagram. Our triangulation will have the following properties:
\begin{enumerate}
\item The number of tetrahedra is bounded above by a linear function of $c$.
\item Each edge of the triangulation is straight in $\mathbb{R}^3$.
\item The meridian of each component of $K \cup L$ and of each edge of $\Gamma$ is simplicial, as is the surgery slope of each component of $L$.
\end{enumerate}
There are many possible ways to build this triangulation. We will follow the recipe given by Coward and the first author in Section 4 of \cite{CowardLackenby}. The triangulation provided by Theorem 4.3 of \cite{CowardLackenby} has all the required properties when $\Gamma = \emptyset$. We only need to generalise to the situation where $\Gamma \not= \emptyset$ and show that the triangulation can be constructed algorithmically in polynomial time, as a function of $c$. We briefly review the steps in Section 4 of \cite{CowardLackenby}.

Step 1 is to embed the underlying planar graph $G$ of the diagram into $\mathbb{R}^2$ as a union of straight arcs, as follows. We first modify $\Gamma$ by expanding each vertex of $\Gamma$ with valence more than $3$ into a tree, so that each of the new vertices has valence exactly $3$. This does not change the exterior of $K \cup \Gamma \cup L$. Let $\overline{G}$ be the graph obtained from $G$ by collapsing parallel edges to a single edge and removing edge loops. F\'ary's theorem says that $\overline{G}$ has an embedding in $\mathbb{R}^2$ where each edge is straight \cite{fary1948straight}; this was proved independently by Wagner \cite{wagner1936bemerkungen} and Stein \cite{stein1951convex} as well. Such an embedding can be found in polynomial time using, for example, the algorithms of \cite{deFPachPollack} or \cite{Schnyder}. We place this embedded graph into the interior of a square $Q$. 
Now reinstate the parallel edges of $G$ with 2 straight arcs each
and the edge loops of $G$ with 3 straight arcs. Step 2 is to replace each edge of $G$ by 4 parallel edges, replace each 2-valent vertex of $G$ by 4 vertices joined by 3 edges and replace each 4-valent vertex of $G$ by $16$ vertices arranged in a grid. Furthermore, each 3-valent vertex of $G$ is replaced by a triangle. This is triangulated by placing a vertex in its interior and coning from that point. The result is a graph $G_+$ where each edge is still straight (see Figure \ref{Fig:Gplus})  .

\begin{figure}[h]
\centering
\includegraphics[width=0.5\textwidth]{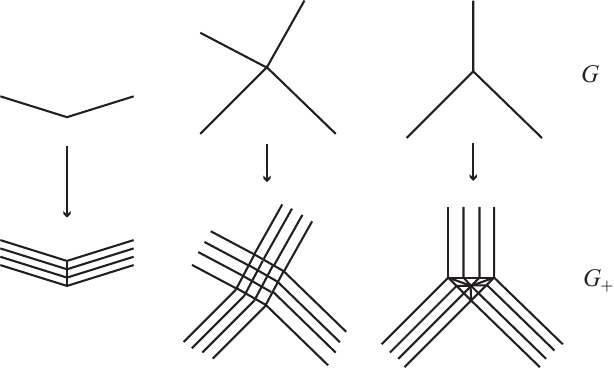}
\caption{Forming $G_+$ from $G$} \label{Fig:Gplus}
\end{figure}

In Step 3, we triangulate the complementary regions of $G_+ \cup \partial Q$ by adding new edges. Let $E$ be its 1-skeleton. In Step 4, we insert 4 copies of $Q$ into the cube $Q \times I$, one being the top face, one the bottom face, and two parallel copies between them. Insert $E \times I$ into the cube $Q \times I$. This divides the cube into convex balls. We triangulate each face of each ball by adding a vertex to its interior and then coning off, and we then triangulate each ball by adding a vertex to its interior and then coning. We now modify this triangulation as follows. Near each crossing, there are $27$ little cubes. We remove these and insert a new triangulation of the cube. A portion of this is shown in Figure \ref{fig:crosstri}. It is clear that if these are inserted in the correct way, we obtain a regular neighbourhood of $K \cup \Gamma \cup L$ as a simplicial subset of this triangulation. Removing the interior of this gives the required triangulation of the exterior of $K \cup \Gamma \cup L$. It is clearly constructible in polynomial time and has the required properties. \\

\begin{figure}[h]
\centering
\includegraphics[width=0.35\textwidth]{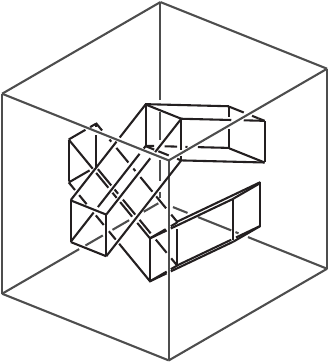}
\caption{The triangulation near each crossing} \label{fig:crosstri}
\end{figure}

\textbf{Step 2}: Building the triangulations of $M - N^\circ(K)$ and $M$.\\

The above  triangulation of $S^3 - N^\circ(K \cup \Gamma \cup L)$ extends to a triangulation of $M - N^\circ(K)$ in an obvious way. We need to attach a solid torus to each component of $\partial N(L)$. We do this by attaching a triangulated meridian disc along the surgery slope, triangulated by adding all the diagonals from an existing vertex on its boundary. We then attach on a 3-ball, which is triangulated as a cone with cone point an existing vertex on its boundary. It is clear that we can extend this triangulation of $M - N^\circ(K)$ to a triangulation of $M$, still using at most $O(c)$ tetrahedra, in which $K$ is simplicial. This process is again completed in polynomial time, and the number of tetrahedra remains bounded above by a linear function of $c$.\\


As described above, we can form a basis for $H^1(M - N^\circ(K))$ by
\begin{enumerate}
\item picking a basis for $H^1(M)$ and taking its image under the homomorphism $H^1(M) \rightarrow H^1(M - N^\circ(K))$ induced by inclusion;
\item adding one extra element that maps to a generator for $H^2(M, M - N^\circ(K))$. 
\end{enumerate}

\vspace{0.5cm}
\textbf{Step 3:} Defining $b_1(M)$ simplicial 1-cocycles on $S^3 - N^\circ(K \cup \Gamma \cup L)$. \\

We have already identified elements of $H^1(M)$  with integral solutions to the equation $A \beta = 0$, where $A$ is the generalised linking matrix for $\Gamma \cup L$. Therefore, consider an integral solution $\beta = (\beta_1 , \cdots , \beta_k)^T$ to the equation $A \beta =0$. The corresponding cocycle $\sum_{i=1}^k \beta_i \psi_i$ is a 1-cocycle on $H^1(S^3 - N^\circ(\Gamma \cup L))$ that evaluates to zero on each surgery slope of the framed link. We can restrict this to a cocycle on $S^3 - N^\circ(K \cup \Gamma \cup L)$, which represents an element of $H^1(M - N^\circ(K))$.

More specifically, define the 1-cocycle $c_{\beta}$ on $S^3 - N^\circ(K \cup \Gamma \cup L)$ as follows. Let $T$ be a maximal tree in the 1-skeleton of the triangulation. For every edge $e \in T$ define $\langle c_{\beta} , e \rangle = 0$. For any oriented edge $e \notin T$, construct a loop $\ell_e$ that starts at the initial vertex of $e$, runs along $e$ and then back to the start of $e$ through an embedded path in $T$. Since we are assigning $0$ to every edge contained in $T$, it should be clear that the numbers assigned to $e$ and $\ell_e$ are the same. Define 
\[ \langle c_{\beta} , e \rangle := \sum_{i=1}^{k} \hspace{1mm} \beta_i \hspace{1mm} \ell k(\ell_e , L_i),  \]
where $\beta_i$ are integers. It is clear that this forms a 1-cocycle since each term $\ell k(\ell_e , L_i)$ is a 1-cocycle. \\

\textbf{Step 4:} Extending the simplicial $1$-cocycles $c_\beta$ to the triangulation of $M - N^\circ(K)$.\\

The triangulation of $M - N^\circ(K)$ is obtained by gluing triangulated solid tori to the triangulation of $S^3 - N^\circ(K \cup \Gamma \cup L)$, such that the restrictions of both triangulations to their common boundary, $\partial N(L)$, agree with each other. The manifold $X$ is obtained by Dehn filling along $L_i$ for $1 \leq i \leq |L|$. 
We can extend the cocycles over the attached solid tori since we started with $\beta$ satisfying $A\beta=0$. This can be achieved with control over the values of newly added edges; see Step 6. It is also easy to see that $\langle c_\beta , m_K \rangle = 0$, where $m_K$ is the meridian of $K$. \\


\textbf{Step 5}: Constructing the extra cocycle.\\

We construct an extra 1-cocycle on $M - N^\circ(K)$ that will form a generator for the summand of $H^1(M - N^\circ(K))$ corresponding to $H^2(M, M - N^\circ(K))$. This extra element, together with the cocycles that formed from a basis for $H^1(M)$, will provide the required basis for $H^1(M - N^\circ(K))$.

Denote by $\kappa = (\kappa_1 , \cdots , \kappa_k)^T$ with $\kappa_i := \ell k (K , L_i)$ the vector encoding the linking numbers of $K$ with $L_i$. We claim that the condition on $K$ being homologically trivial in $M$ is equivalent to the linear equation $\theta A = - \kappa$ having an integral solution for $\theta$. The homology group $H_1(S^3 - N^\circ(\Gamma \cup L)  ; \mathbb{Z})$ is freely generated by the meridians $\mu_1 , \cdots , \mu_k$ encircling $L_1, \cdots, L_k$. For  $1 \leq i \leq |L|$, denote by $\lambda_i$ the longitude of $L_i$ that has zero linking number with $L_i$. Then $H_1(M; \mathbb{Z})$ is obtained by adding the relations $a_{ii} \hspace{1mm} \mu_i + \lambda_i =0$, one for each component  $L_i$ of $L$. The latter relation is equivalent to
\[ a_{ii} \hspace{1mm} \mu_i + \sum_{j \neq i} \ell k (L_i , L_j) \mu_j = \sum_j a_{ij} \hspace{1mm} \mu_j = 0. \]
Therefore, $K$ being trivial in $H_1(M  ; \mathbb{Z})$ is exactly the condition that $K$ is equal in $H_1(S^3 - N^\circ(\Gamma \cup L)  ; \mathbb{Z})$ to some integral linear combination of these relations. This is exactly the condition that $\theta A = -\kappa$ has an integral solution for $\theta$.

Let $\theta$ be any integral solution to the linear equation $\theta A = - \kappa$. Define the 1-cocycle $c_\theta$ similar to $c_\beta$ but with a slight modification to make the evaluation on the meridian of $K$ non-zero. More precisely  
\[ \langle c_\theta , e \rangle :=  \ell k (\ell_e , K) + \sum_{i=1}^{|L|} \hspace{1mm} \theta_i \hspace{1mm} \ell k(\ell_e , L_i).\]
The evaluation of $c_\theta$ on each surgery curve is zero, and the evaluation on the meridian of $K$ is equal to 1. It therefore is sent to the generator of $H^2(M, M - N^\circ(K))$, under the map $H^1(M - N^\circ(K)) \rightarrow H^2(M, M - N^\circ(K))$.\\

\textbf{Step 6:} Analysing the computational cost of the algorithm. \\

The number of edges of the triangulation of $S^3 - N^\circ(K \cup \Gamma \cup L)$ is $O(c)$. A spanning tree $T$ and the loops $\ell_e$ for $e \notin T$ can be found in polynomial time in the number of edges. The numbers $\ell k (\ell_e , L_i)$ can be computed as follows. We can construct the diagram $L \cup \Gamma \cup \ell_e$, since $\ell_e$ is a union of edges of the triangulation. Each edge of the triangulation is straight, and so when it is projected to the plane of the diagram, the image of $\ell_e$ is a union of straight arcs. We compute the linking number $\ell k (\ell_e , L_i)$ using the usual signed count over the crossings of $\ell_e$ with $L_i$. Each of the linking numbers is at most linear in the number of crossings $c$. This is because the triangulation of $S^3 - N^\circ(K \cup \Gamma \cup L)$ has $O(c)$ edges and each edge can contribute at most a constant number of crossings. Moreover the coordinates $\theta_i$ are at most linear in $c$ and can be computed in polynomial time, as $A$ is a fixed matrix. Therefore the evaluation of each constructed $1$-cocycle of $S^3 - N^\circ(K \cup \Gamma \cup L)$ on each edge has absolute value $O(c^2)$ and can be computed in polynomial time. Recall that the triangulation of $S^3 - N^\circ(K \cup \Gamma \cup L)$ is extended to the attached solid tori by first attaching discs whose boundaries are the simplicial surgery slopes, and then attaching solid balls. The attached discs and solid balls are themselves triangulated as the cone over one of their vertices. We can extend each $1$-cocycle over the attached discs by keeping its evaluation on each edge $O(c^3)$ in absolute value, since each edge of the relevant triangulated disc is homologous to a sum of $O(c)$ edges on its boundary. More generally, each $1$-cocycle on $S^3 - N^\circ(K \cup \Gamma \cup L)$ can be extended to the attached solid tori while keeping its evaluation on each edge $O(c^3)$ in absolute value. This is because each edge of an attached solid torus is homologous to a sum of $O(c)$ edges on its boundary torus. Extending these $1$-cocycles over the attached triangulated solid tori can be done in polynomial time. Moreover, the extension keeps the total number of tetrahedra linear in $c$. Since the evaluation of each $1$-cocycle on each edge has absolute value $O(c^3)$ and there are $O(c)$ edges, so $C_{\mathrm{una}}$ of each $1$-cocycle is $O(c^4)$. As $b_1(X) = b_1(M)+1$ is a fixed constant, the total $C_{\mathrm{una}}$ over all $1$-cocycles is also $O(c^4)$. 
\end{proof}

\section{Surfaces with trivial Thurston norm}



Recall from Section \ref{Sec:Normal} the definition of a fundamental normal surface. In this section, we will prove the following result.

\begin{thm}
	\label{Thm:FundamentalNormZero}
	Let $\mathcal{T}$ be a triangulation of a compact orientable irreducible 3-manifold $X$. If $X$ has any compressible boundary components, suppose that these are tori. The subspace of $H_2(X, \partial X; \mathbb{Z})$ with trivial Thurston norm is spanned by a collection of fundamental normal tori, annuli and discs.
\end{thm}

As an immediate consequence, we obtain the following.

\theoremstyle{theorem}
\newtheorem*{Thm:BasisForW}{Theorem \ref{Thm:BasisForW}}
\begin{Thm:BasisForW}
	Let $\mathcal{T}$ be a triangulation of a compact orientable irreducible 3-manifold $X$. If $X$ has any compressible boundary components, suppose that these are tori. Then there is a collection $w_1, \cdots, w_r$ of integral cocycles that forms a basis for the subspace $W$ of $H^1(X; \mathbb{R})$ consisting of classes with Thurston norm zero with $\sum_i C_{\mathrm{dig}}(w_i)$ at most $O(|\mathcal{T}|^3)$.
\end{Thm:BasisForW}

\begin{proof}
	By Theorem \ref{Thm:FundamentalNormZero}, there is a collection of fundamental normal surfaces that forms a generating set for $W \cap H^1(X; \mathbb{Z})$. Some subset of this collection therefore forms a basis for $W$. Each normal coordinate of the fundamental normal surface $S$ is at most $t2^{7t+2}$, where $t = |\mathcal{T}|$, the number of tetrahedra of $\mathcal{T}$. This implies that $S$ intersects each edge of $\mathcal{T}$ at most $3t 2^{7t+2}$ times, since for a tetrahedron $\Delta$ containing the edge there are at most $3$ different normal disc types intersecting that edge. 
	Hence, when $S$ is oriented, the cocycle $w$ dual to $S$ has evaluation at most $3t2^{7t+2}$ on each edge. So $C_{\mathrm{dig}}(w)$ is at most $O(t^2)$. Since the number of $w_i$ is $O(t)$, the total $C_\mathrm{dig}$ is $O(t^3)$.
\end{proof}

We will prove Theorem \ref{Thm:FundamentalNormZero} using results of Tollefson and Wang \cite{TollefsonWang}. In that paper, $X$ was required to be irreducible and its compressible boundary components were required to be tori. It is for this reason that these are also hypotheses of Theorem \ref{Thm:FundamentalNormZero}.

\begin{definition}
Let $X$ be a compact orientable 3-manifold with a triangulation $\mathcal{T}$.
	A compact incompressible $\partial$-incompressible oriented normal surface $F$ properly embedded in $X$ is \emph{lw-taut} if
	\begin{enumerate}
		\item its homology class $[F]$ is non-trivial in $H_2(X, \partial X)$;
		\item it is $\chi_-$ minimising;
		\item there is no union of components of $F$ that is homologically trivial;
	\end{enumerate}
	and furthermore it has smallest weight among all incompressible $\partial$-incompressible normal surfaces in its homology class satisfying the above conditions.
\end{definition}



Recall from Section \ref{Sec:Normal} the definition of the projective solution space ${\mathcal P}(\mathcal{T})$. Recall also the notion of a normal surface being carried by a face of $\mathcal{P}(\mathcal{T})$.

\begin{definition} A face of $\mathcal{P}(\mathcal{T})$ is \emph{lw-taut} if every surface carried by the face is lw-taut.
\end{definition}

The following result was proved by Tollefson and Wang (Theorem 3.3 and Corollary 3.4 in \cite{TollefsonWang}).

\begin{thm} 
	\label{Thm:lwtaut}
	Let $\mathcal{T}$ be a triangulation of a compact orientable irreducible 3-manifold $X$. If $X$ has any compressible boundary components, suppose that these are tori. Let $F$ be an lw-taut surface and let $C$ be the minimal face of $\mathcal{P}(\mathcal{T})$ that carries $F$. Then $C$ is lw-taut. Furthermore, there are unique orientations assigned to the surfaces carried by $C$ such that if $G$ and $H$ are carried by $C$, then the normal sum $G+H$ satisfies $[G+H] = [G] + [H] \in H_2(X, \partial X)$ and $x([G+H]) = x([G]) + x([H])$.
\end{thm}

\begin{proof}[Proof of Theorem \ref{Thm:FundamentalNormZero}] Let $T$ consist of those fundamental annuli, tori and discs that lie in some lw-taut face. Consider any element of $H_2(X, \partial X; \mathbb{Z})$ with trivial Thurston norm. This is represented by an lw-taut surface $F$. Let $C$ be the minimal face of $\mathcal{P}(\mathcal{T})$ that carries $F$.  By Theorem \ref{Thm:lwtaut}, $C$ is an lw-taut face. Now, $F$ is a normal sum of fundamental surfaces $G_1,  \cdots, G_n$ that are also carried by $C$. By Theorem \ref{Thm:lwtaut}, they are all oriented surfaces. Since they are lw-taut and $X$ is irreducible, no $G_i$ is a sphere. By Theorem \ref{Thm:lwtaut}, $[F] = [G_1] + \cdots + [G_n]$ in $H_2(X, \partial X)$ and $0 = x([F]) = x([G_1]) + \cdots + x([G_n])$. Since Thurston norm is always non-negative, this implies that $x([G_i]) = 0$ for each $i$. As the $G_i$ are oriented and lw-taut, they are discs, annuli and tori, and hence they are elements of $T$.
\end{proof}

\section{Computational complexity of \textsc{Thurston norm ball}}

In this section, we analyse the decision problem \textsc{Thurston norm ball for 3-manifolds with $b_1 \leq B$} that was mentioned in the Introduction. We now define it precisely. The input is a triangulation $\mathcal{T}$ for a compact orientable 3-manifold $X$ with first Betti number $b_1(X) \leq B$, and a collection of integral simplicial 1-cocycles $\{ \phi_1 , \cdots, \phi_b \}$ that forms a basis for $H^1(X ; \mathbb{R})$. The problem asks to compute the unit ball for the Thurston semi-norm. Here we have identified $H^1(X  ; \mathbb{R})$ with $H_2(X , \partial X ; \mathbb{R})$ using Poincar\'{e} duality. The output consists of the following two sets of data:\\

1) A collection of integral cocycles that forms a basis for the subspace $W$ of $H^1(X; \mathbb{R})$ with Thurston norm zero. These are written as rational linear combinations of the given cocycles $\{ \phi_1 , \cdots , \phi_b \}$. Denote by $p$ the projection map from $H^1(X ; \mathbb{R})$ to $H^1(X ; \mathbb{R})/W$.\\

2) A finite set of points $V \subset H^1(X ; \mathbb{Q})$ such that $p(V)$ is the set of vertices of the unit ball of $H^1(X ; \mathbb{R})/W$, together with a list $\mathcal{F}$ of subsets of $V$. The set $\mathcal{F}$ is the list of faces of the unit ball for $H^1(X ; \mathbb{R})/W$. In other words, for $F\in \mathcal{F}$ the set
\[ \{ p(v) \hspace{2mm} | \hspace{2mm} v \in F \}, \]
forms the set of vertices of some face of the unit ball for $H^1(X ; \mathbb{R})/W$. Moreover, this covers all the faces as we go over all the elements of $\mathcal{F}$. Thus, the unit ball of $H^1(X; \mathbb{R})$ is the inverse image of the unit ball of $H^1(X ; \mathbb{R})/W$ under the projection map $p$.

The complexity of the input is defined to be $|\mathcal{T}| + \sum_i C_{\mathrm{una}}(\phi_i)$. Recall that $|\mathcal{T}|$ is the number of tetrahedra of $\mathcal{T}$. As discussed in the Introduction, the fact that the complexity of $\phi_i$ is measured using $C_{\mathrm{una}}$ rather than $C_{\mathrm{dig}}$ is definitely not standard. In order to simplify the notation a little, we let $\Phi$ be the matrix with columns $\phi_1, \cdots, \phi_b$. More specifically, it has $b$ columns and has a row for each oriented edge of $\mathcal{T}$, and its $(i,j)$ entry is the evaluation of $\phi_j$ on the $i$th edge. So $C_{\mathrm{una}}(\Phi) =  \sum_i C_{\mathrm{una}}(\phi_i)$.



\theoremstyle{theorem}
\newtheorem*{thurston ball}{Theorem \ref{thurston ball}}
\begin{thurston ball}
Fix an integer $B \geq 0$. The problem \textsc{Thurston norm ball for 3-manifolds with $b_1 \leq B$} lies in \textbf{FNP}, where $b_1$ denotes the first Betti number.
\end{thurston ball}



We will prove this over the next two sections. In this section, we will consider the following restricted version of the problem.

In \textsc{Thurton norm ball for irreducible boundary-irreducible 3-manifolds with $b_1 \leq B$}, we consider compact, orientable, irreducible, boundary-irreducible 3-manifolds $X$. We allow $X$ to be disconnected. Thus, the input is a triangulation $\mathcal{T}$ for $X$ with first Betti number $b_1(X) \leq B$, and a collection of simplicial integral 1-cocycles $\{ \phi_1 , \cdots, \phi_b \}$ that forms a basis for $H^1(X; \mathbb{R})$.  The output is the data in (1) and (2) above.

\begin{thm}
\label{Thm:MainThmIrred}
\textsc{Thurton norm ball for irreducible boundary-irreducible 3-manifolds with $b_1 \leq B$} is in \textbf{FNP}.
\end{thm}

\begin{proof}
Let $d = \dim (H^1(X; \mathbb{R})/W)$, and denote by $B_{\bar{x}}$ the unit ball of the induced Thurston norm $\overline{x}$ on $H^1(X; \mathbb{R})/W$: 
\[B_{\bar{x}} = \{ v \in H^1(X ; \mathbb{R})/W \hspace{3mm} | \hspace{3mm} \overline{x}(v)\leq 1 \}. \]
Then $B_{\bar{x}}$ is a convex polyhedron. The boundary, $\partial B_{\bar{x}}$, inherits a facial structure from $B_{\bar{x}}$, where the faces of $\partial B_{\bar{x}}$ correspond to faces of $B_{\bar{x}}$ except for the face $B_{\bar{x}}$ itself. In particular, top-dimensional faces of $\partial B_{\bar{x}}$ correspond to facets of $B_{\bar{x}}$, and from now on a top-dimensional face refers to a top-dimensional face of $\partial B_{\bar{x}}$. The plan of the proof is as follows: 
\begin{enumerate}
	\item A basis for the subspace $W$ consisting of classes with Thurston norm zero is given to us non-deterministically. 
	\item The list of vertices $V$ and faces $\mathcal{F}$ is given to us non-deterministically.
	\item We verify that for each element $F \in \mathcal{F}$, the vertices of $F$ actually lie on the same face of $\partial B_{\bar{x}}$. 
	\item Let $P$ be the space obtained by patching together geometric realisations of given top-dimensional elements of $\mathcal{F}$ along their common boundaries. We have the maps 
	\[ P \xrightarrow{i} \partial B_{\bar{x}} \xrightarrow{\pi}S^{d-1}, \]
	where $i$ is the inclusion (well-defined by (3)) and $\pi$ is the radial projection onto the $(d-1)$-dimensional sphere $S^{d-1}$. We verify that the composition $\pi \circ i$, and hence $i$, is a homeomorphism. In the next two steps, we verify that the cell structure of $P$ agrees with the facial structure of $\partial B_{\bar{x}}$. 
	\item We verify that each element of $\mathcal{F}$ (including the elements of $V$) is sent by the map $i$ to a face of $\partial B_{\bar{x}}$.
	\item We verify that the list of faces of $\partial B_{\bar{x}}$ is equal to $\mathcal{F}$.
\end{enumerate}

\textbf{Step 1: A basis for $W$}\\
By Theorem \ref{Thm:BasisForW}, there is a collection $w_1, \cdots, w_r$ of integral cocycles that forms a basis for the subspace $W$ of $H^1(X; \mathbb{R})$ consisting of classes with Thurston norm zero and that satisfies $\sum_i C_{\mathrm{dig}}(w_i) \leq O(|\mathcal{T}|^3)$. We assume that these simplicial cocycles are given to us non-deterministically. We can certify that the elements $w_1, \cdots, w_r$ have Thurston norm zero, using Theorem \ref{lackenby}. 

We express each $w_i$ as a linear combination of the given cocycles $\phi_1, \cdots, \phi_b$, as follows. There is a coboundary map $\partial^\ast \colon C^0(\mathcal{T}) \rightarrow C^1(\mathcal{T})$ from 0-cochains to 1-cochains. There is a natural basis $x_1, \cdots, x_m$ for $C^0(\mathcal{T})$ where $x_i$ is the 0-cochain that evaluates to $1$ on the $i$th vertex of $\mathcal{T}$ and evaluates to zero on the other vertices. We wish to solve
$$\alpha_1 \phi_1 + \cdots + \alpha_b \phi_b + \beta_1 \partial^\ast (x_1) + \cdots + \beta_m \partial^\ast (x_m) = w_i.$$
Using the Bareiss algorithm, this can be done in polynomial time as a function of $C_{\mathrm{dig}}(\Phi)$ and $|\mathcal{T}|$. The resulting coefficients $\alpha_1, \cdots, \alpha_b$ have $C_{\mathrm{dig}}(\alpha_i)$ at most a polynomial function of $C_{\mathrm{dig}}(\Phi)$ and $|\mathcal{T}|$.
We can also verify whether the cocycles $w_1, \cdots, w_r$ are linearly independent in $H^1(X; \mathbb{R})$.

In the remaining steps, we will certify that the induced Thurston semi-norm on $H^1(X; \mathbb{R})/W$ is indeed a norm, hence the basis elements actually generate $W$.\\

\textbf{Step 2A: Bounding the number of faces and vertices of the Thurston unit ball}\\
We are given the simplicial integral cocycles $\{ \phi_1 , \cdots, \phi_b \}$. From these, we can construct properly embedded oriented surfaces $S_1 , \cdots , S_b$ that are Poincar\'{e} dual to $\phi_1 , \cdots , \phi_b$, and whose total complexity, $\sum_i \chi_-(S_i)$, is at most $O(C_{\mathrm{una}}(\Phi))$. To see this geometrically, fix a 1-cocycle $\phi_i$ and consider an arbitrary simplicial triangle $\Delta$ in the triangulation. Assume that the numbers that $\phi_i$ associates to the edges of $\Delta$ are $a, b , c \geq 0$ such that $a = b+c$. We can draw $a= b+c$ normal arcs in $\Delta$ that intersect the edges of $\Delta$ in respectively $a, b $ and $c$ points. Given any tetrahedron, we can look at the drawn normal curves on its boundary triangles and place normal disks (triangles or squares) inside the tetrahedron with the given boundary curves. Construct an embedded surface $S_i$ by putting together the normal disks together glued along the common boundaries of the tetrahedra. The constructed surface is Poincar\'{e} dual to the starting 1-cocycle, and $\chi_-(S_i)$ is at most a linear multiple of $C_{\mathrm{una}}(\phi_i)$. 

By Theorem \ref{number of faces}, the total number of faces and vertices of the Thurston unit ball for $H^1(X ; \mathbb{R})/W$ are at most polynomial functions of $C_{\mathrm{una}}(\Phi)$. Note the degrees of these polynomials are bounded above by $B^2$, which is a fixed constant by our assumption. 

\begin{remark}
	Here the use of unary complexity $C_{\mathrm{una}}$ is crucial for our argument. 
\end{remark}


\textbf{Step 2B: Bounding the number of bits encoding the coefficients of the vertices of the Thurston unit ball}

\begin{lem}
There is a set of points $V \subset H^1(X ; \mathbb{Q})$ such that 
\[ \{ p(v) \hspace{2mm} | \hspace{2mm} v \in V \},  \]
is the set of vertices of the unit ball for the Thurston norm on $H^1(X , \mathbb{R})/W$ with the following properties:
\begin{enumerate}
\item $|V|$ is at most a polynomial function of $C_{\mathrm{una}}(\Phi)$;
\item each element of $V$ is $\gamma_1\phi_1 + \cdots + \gamma_b \phi_b$, for rational numbers $\gamma_1, \cdots, \gamma_b$ such that $\sum_i C_{\mathrm{dig}}(\gamma_i)$ is at most a polynomial function of $\log(C_{\mathrm{una}}(\Phi)))$.
\end{enumerate}
\label{vertex coefficients}
\end{lem}
\begin{proof}
Item (1) is proved in Step 2A. 

Define $\mathcal{A}$ as the set of integral points in $H^2(X , \partial X ; \mathbb{Z}) \otimes \mathbb{Q}$ with dual norm at most one. By the previous step, we can construct surfaces $S_1 , \cdots , S_b$ Poincar\'{e} dual to $\phi_1 , \cdots , \phi_b$ whose total complexity, $\sum \chi_-(S_i)$, is at most $O(C_{\mathrm{una}}(\Phi))$. By Theorem \ref{lattice points}, the size of $\mathcal{A}$ is at most a polynomial function of $O(C_{\mathrm{una}}(\Phi))$. Let $v \in H^1(X ; \mathbb{Q})$ be such that $p(v)$ is a vertex of the unit ball for  $H^1(X ; \mathbb{R})/W$. Then there are points $a_1 , \cdots , a_r \in \mathcal{A}$ such that the set of points $z \in H^1(X ; \mathbb{R})$ satisfying the equations 
\[  \langle a_1 , PD(z) \rangle = 1 , \]
\[ \vdots \]
\[ \langle a_r , PD(z) \rangle =1, \]
coincides with the affine space $v +W$. Here $PD(z)$ is the Poincar\'{e} dual to $z$, and $a_1 , \cdots, a_r$ can be chosen to be the set of vertices spanning the face of the dual unit ball that is dual to the vertex $p(v)$. Moreover, since $z \in H^1(X ; \mathbb{R})$ lies inside a $b$-dimensional space, at most $b$ of the above equations can be linearly independent; hence we may assume that $r \leq b$ by choosing a suitable subset of $\{ a_1 , \cdots, a_r \}$. Recall that the dual basis $\{ e^1 , \cdots , e^b \}$ for $H^2(X , \partial X ; \mathbb{Z}) \otimes \mathbb{Q}$ is defined as 
\[ \langle e^i , [S_j] \rangle = \delta_{ij}, \]
where $\delta_{ij}$ is the Kronecker function. From the proof of Theorem \ref{lattice points} we know that, if we write $a_i$ in the basis $\{ e^1 , \cdots , e^b \}$ then the coefficients are integral, and their absolute values are bounded above by $O(C_{\mathrm{una}}(\Phi))$. Hence for each $1 \leq i \leq  r$ we can write
\[ a_i = \sum_j \eta_j^i \hspace{1mm} e^j,\]
with $|\eta_j^i| \leq O(C_{\mathrm{una}}(\Phi))$. Since $ \{ \phi_1 , \cdots , \phi_b \}$ is a basis for $H^1(X ; \mathbb{R})$ we can write 
\[ z = \gamma_1 \phi_1 + \cdots + \gamma_b \phi_b  , \]
for real numbers $\gamma_j$. Now for $1 \leq i \leq r$ we have
\[ 1 = \langle a_i , PD(z) \rangle  = \langle  \sum_j \eta_j^i \hspace{1mm} e^j , \sum_s \gamma_s [S_s] \rangle = \eta_1^i \gamma_1 + \cdots + \eta_b^i \gamma_b.  \]
This gives a set of $r$ linear equations for $\gamma_1 , \cdots , \gamma_b$. The number of variables and the number of equations are bounded above by the constant $B$, and the total absolute values of the coefficients $\eta_j^i$ is bounded above by a polynomial function of $C_{\mathrm{una}}(\Phi)$. Therefore, there exists a rational solution, $z = \gamma_1 \phi_1 + \cdots + \gamma_b \phi_b$, where the total number of bits for $(\gamma_1 , \cdots , \gamma_b)$ is at most a polynomial function of $\log(C_{\mathrm{una}}(\Phi))$, for example by the Bareiss algorithm. 
\end{proof}


\textbf{Step 2C: The list of vertices, $V$, and faces $\mathcal{F}$}\\
By Lemma \ref{vertex coefficients}, there is a set of points $V \subset H^1(X ; \mathbb{Q})$ such that 
\[ \{ p(v) \hspace{2mm} | \hspace{2mm} v \in V \},  \]
is the set of vertices of the unit ball for the Thurston norm on $H^1(X , \mathbb{R})/W$, $|V|$ is at most a polynomial function of $C_{\mathrm{una}}(\Phi)$, and the total number of bits for writing each element of $V$ in terms of the basis $\{ \phi_1 , \cdots , \phi_b \}$ is at most a polynomial function of $\log(C_{\mathrm{una}}(\Phi))$. Likewise, the number of faces of the unit ball for $H^1(X ; \mathbb{R})/W$, that is $|\mathcal{F}|$, is bounded above by a polynomial function of $C_{\mathrm{una}}(\Phi)$. The sets $V$ and $\mathcal{F}$ are part of the certificate, and are given to us non-deterministically. We use Theorem \ref{lackenby} to certify that each element of $V$ has Thurston norm one.\\

\textbf{Step 3: Certifying that the vertices of each element of $\mathcal{F}$ actually lie on the same face of $\partial B_{\bar{x}}$}\\
Recall that the number of elements of $\mathcal{F}$ is at most polynomially large in $C_{\mathrm{una}}(\Phi)$. For any element $F \in \mathcal{F}$ with $F = \{ u_1 , \cdots , u_s \}  $, we check that 
\[ x(K_1u_1 + \cdots + K_s u_s) = K_1 \hspace{1mm}x(u_1)+ \cdots + K_s \hspace{1mm} x(u_s) = K_1 + \cdots + K_s, \] 
for some positive integral choices of $K_1 , \cdots, K_s$. Here $x$ represents the Thurston norm on $H^1(X ; \mathbb{R})$. It is clear that once proven, this implies that $\{u_1 , \cdots , u_s \}$ lie on the same face.

We would like to choose each $K_i$ such that the $1$-cocycle $K_i  u_i$ is integral. First we need to check that there is a choice of $K_i$ that is not too large. To see this, we can write $u_i$ in the integral basis $\phi_1 , \cdots, \phi_b$ as 
\[ u_i = \alpha_1^i \hspace{1mm} \phi_1 + \cdots + \alpha_b^i \hspace{1mm} \phi_b, \]
and take $K_i$ to be the product of denominators of $\alpha_j^i$ for $1 \leq j \leq b$. From Step 2 we know that the total number of digits of $K_i$ is bounded above by a polynomial function of $C_{\mathrm{una}}(\Phi)$. The numbers $K_i$ are part of the certificate, and are given to us non-deterministically. 

Therefore, we can use Theorem \ref{lackenby} to certify that the Thurston norm of
\[ K_1 \hspace{1mm} u_1 + \cdots + K_s \hspace{1mm} u_s, \]
is $K_1 + \cdots + K_s$. This finishes the certification for each element $F \in \mathcal{F}$. Since the total number of elements of $\mathcal{F}$ is bounded above by a polynomial function of $C_{\mathrm{una}}(\Phi)$, we are done.\\

\textbf{Step 4A: Decomposing each top-dimensional element in $\mathcal{F}$ into simplices.}\\
The dimension of an element $F$ of $\mathcal{F}$ is the maximum number $m$ such that $F$ has $m+1$ affinely independent vertices. Hence, we can compute the dimension of each such element from the list of its vertices, in time that is bounded above by a polynomial function of $C_{\mathrm{una}}(\Phi)$. The boundary of the polytope $\partial B_{\bar{x}}$  can be subdivided to a triangulation without adding any new vertices. This can easily be proved by induction on the dimension of the faces. In particular, each top-dimensional element in $\mathcal{F}$ can be subdivided in this way so that along incident subsets, their triangulations agree. Such a subdivision will be provided to us non-deterministically. Thus, for each top-dimensional element $F$ in $\mathcal{F}$, we are provided with a collection of subsets of $F$, each consisting of $d$ vertices, where $d$ is the dimension of $H^1(X ; \mathbb{R}) / W$. The number of such subsets is at most $|F|^d$, which is at most $|V|^b$. Let $\Sigma$ denote the collection of all these subsets, as we run over all top-dimensional elements of $\mathcal{F}$. Then the number of elements of $\Sigma$ is at most a polynomial function of $C_{\mathrm{una}}(\Phi)$. \\


\textbf{Step 4B: Certifying that the composition $\pi \circ i$ is injective.}\\
Let $\hat{\Sigma}$ be the set of faces of elements of $\Sigma$; in particular $\Sigma \subset \hat{\Sigma}$. Moreover, their cardinalities satisfy $|\hat{\Sigma}| \leq 2^d |\Sigma|$, since each element of $\Sigma$ has $d$ vertices. Recall that $P$ is the abstract CW-complex obtained by taking the geometric realisations of elements of $\mathcal{F}$ and gluing them together along their common boundaries. It is enough to verify that any pair of simplices $\sigma_1, \sigma_2 \in \hat{\Sigma}$ have disjoint interiors.
Here and afterwards, we slightly abuse the notation by denoting the image under $i$ of the geometric realisation of $\sigma_j$ by $\sigma_j$ again. The condition $\sigma_1^\circ \cap \sigma_2^\circ = \emptyset$ is equivalent to $\text{Cone}(\sigma_1^\circ) \cap \text{Cone}(\sigma_2^\circ) =\emptyset$, since the restriction of the Thurston norm to both $\sigma_1$ and $\sigma_2$ is equal to $1$. We now show how to verify the condition $\text{Cone}(\sigma_1^\circ) \cap \text{Cone}(\sigma_2^\circ) =\emptyset$ using Mixed Integer Programming. Assume that $\{ u_1 , \cdots , u_r \}$ forms the list of vertices of $\sigma_1$, and $\{ y_1 , \cdots , y_s \}$ forms the list of vertices of $\sigma_2$. We would like to check that 
\begin{eqnarray}
\alpha_1 u_1 + \cdots + \alpha_r u_r = \beta_1 y_1 + \cdots + \beta_s y_s
\label{disjoint-cones}
\end{eqnarray}  
has no solution for $\alpha_i > 0$ and $\beta_j > 0$. However, Equation (\ref{disjoint-cones}) has a solution with $\alpha_i >0$ and $\beta_j>0$ if and only if it has a solution with $\alpha_i\geq 1$ and $\beta_j \geq 1$, essentially by scaling. This is an instance of Mixed Integer Programming. Since no variables are required to be integers, it can be solved in polynomial time as a function of $\log(C_{\mathrm{una}}(\Phi))$ by Lemma \ref{vertex coefficients}.\\

\textbf{Step 4C: Certifying that the composition $\pi \circ i$ is a homeomorphism.}\\
The certificate provides the vertices and faces of the boundary of the unit norm ball, and it provides a collection $\Sigma$ of $(d-1)$-dimensional simplices. We have checked that the interiors of these simplices are disjoint. But we need to check that their union is the entire boundary of the unit norm ball. We check that $P$ is a closed, oriented, pseudo-manifold as follows. For this purpose, we need to check the conditions of \emph{purity}, \emph{non-branching}, \emph{connectivity} and \emph{orientability}.


\emph{Purity condition}: For each element $\sigma$ of $\Sigma$, we check that $p(\sigma)$ actually forms the vertices of a $(d-1)$-dimensional simplex. To do this, we verify that its vertices form a linearly independent set in $H^1(X ; \mathbb{R}) / W$. This can be done in polynomial time in $|\mathcal{T}|$ and $C_{\mathrm{una}}(\Phi)$.


\emph{Non-branching condition}: We check that every $(d-2)$-dimensional simplex appears in exactly two $(d-1)$-dimensional simplices. In other words, for each $(d-2)$-dimensional face of $\Sigma$, we check that it lies in exactly one other $(d-1)$-dimensional simplex in $\Sigma$. Since $|\Sigma|$ is bounded by a polynomial in $C_{\mathrm{una}}(\Phi)$, this can be checked in time that is polynomially bounded in $C_{\mathrm{una}}(\Phi)$.  

\emph{Connectivity condition}: For every pair of simplices $\sigma_1$ and $\sigma_2$ in $\Sigma$, we check that $\sigma_1$ and $\sigma_2$ can be connected by a path consisting of $(d-2)$- and $(d-1)$-dimensional simplices. We may assume that such a (minimal) path is given to us non-deterministically. Since $|\Sigma|$ is bounded by a polynomial in $C_{\mathrm{una}}(\Phi)$, this can be checked in time that is polynomially bounded in $C_{\mathrm{una}}(\Phi)$.

\emph{Orientability}: We specify an orientation of each simplex in $\Sigma$ by specifying an ordering of its vertices. We check that this orientation is compatible with its orientation from $H^1(X ; \mathbb{R}) / W$, by checking that the matrix with columns given by the elements of $\Sigma$ has positive determinant. For every two top-dimensional faces that share a $(d-2)$-dimensional face, we check that they are glued by an orientation-reversing map along their intersection.

We have now established that $P$ is a closed, oriented, pseudo-manifold. We check that the map $\pi \circ i$ is injective and surjective, and hence a homeomorphism. Injectivity was established in Step 4B. To prove the surjectivity, it is enough to show that the degree of the map $\pi \circ i$ is non-zero. Here we are using the fact that the degree is well-defined between compact, oriented pseudo-manifolds of the same dimension. Moreover, any such map that is not surjective has degree $0$, since $S^{d-1} - \{ \text{point} \}$ is contractible and the degree is invariant under homotopy. To check that the degree is non-zero in our case, note that the degree can be computed as the signed count of the points in $P$ that map to a generic but fixed element in $S^{d-1}$. Since all of the signs agree by our construction, this signed count is always non-zero. This finishes Step 4 of the certification. 

The maps $\pi \circ i$ and $\pi$ are both homeomorphisms, hence so is the map $i \colon P \rightarrow \partial B_{\bar{x}}$. Next we need to verify that the cell structure on $P$ agrees with the facial structure of $\partial B_{\bar{x}}$.\\

\textbf{Step 5: Certifying that each element of $\mathcal{F}$ is sent by the map $i$ to a face of $\partial B_{\bar{x}}$}\\
Note that each element of $V$ is an element of $\mathcal{F}$ as well. In particular, once proven, it implies that each element of $V$ is sent by the map $i$ to a vertex of $\partial B_{\bar{x}}$. We use the following equivalent definition of a face \cite[page 16]{schneider_2013}: Given a convex set $A \subset \mathbb{R}^d$, a convex subset $F$ of $A$ is a face if $y,z \in A$ and $(y+z)/2 \in F$ implies that $y, z \in F$. Note that $B_{\bar{x}}$ is the convex hull of $\partial B_{\bar{x}} = i(P)$, which is equal to the convex hull of $i(V)$ by the construction of $P$. We want to verify that $\mathrm{Conv}(i(F))$ is a face of $B_{\bar{x}}$, where $\mathrm{Conv}(\cdot)$ denotes the convex hull. By the above definition, $\mathrm{Conv}(i(F))$ is \emph{not} a face of $B_{\bar{x}}$ if and only if there are $y,z \in B_{\bar{x}}$ such that $(y+z)/2 \in \mathrm{Conv}(i(F))$ but $y \notin \mathrm{Conv}(i(F))$. Note that $y, z \in B_{\bar{x}}$ have norm at most one, and if $(y+z)/2$ has norm equal to one then necessarily both $y$ and $z$ have norm equal to one. Therefore, $\mathrm{Conv}(i(F))$ is not a face of $B_{\bar{x}}$, if and only if there are $y,z \in \partial B_{\bar{x}}$ such that $(y+z)/2 \in \mathrm{Conv}(i(F))$ but $y \notin \mathrm{Conv}(i(F))$. Now $\partial B_{\bar{x}}= i(P)$ and a CW-complex is the disjoint union of interiors of its cells. So $y \in \partial B_{\bar{x}} - \mathrm{Conv}(i(F))$ if and only if $y$ lies in the interior of $\mathrm{Conv}(i(F'))$ for some $F' \in \mathcal{F}$ such that $F' \nsubseteq F$. 

Let $V = \{v_1, \cdots, v_k \}$ and $F = \{ u_1, \cdots, u_r\}$. We have shown that $\mathrm{Conv}(i(F))$ is not a face of $B_{\bar{x}}$ if and only if there is an element $F' \nsubseteq F$ with $F' = \{ y_1, \cdots, y_s \}$ such that the following system of inequalities has a solution
\begin{align}
	& y = \alpha_1 y_1 + \cdots + \alpha_s y_s; \\
	& z = \beta_1 v_1 + \cdots + \beta_k v_k; \\
	&(y+z)/2 = \gamma_1 u_1 + \cdots + \gamma_r u_r;\\
	& \alpha_i>0 ;  \\
	&\beta_j, \gamma_k \geq 0; \\
	& \sum \alpha_i = \sum \beta_j = \sum \gamma_k = 1.	
\end{align}

By scaling $y$ and $z$ by the same positive factor, we see that the above system of inequalities has a solution if and only if the following system has a solution 
\begin{align*}
	&(5)-(7); \\
	& \alpha_i \geq 1; \\
	& \beta_j, \gamma_k \geq 0; \\
	&  \sum \alpha_i = \sum \beta_j = \sum \gamma_k.
\end{align*}

This is an instance of Mixed Integer Programming (with no integer variables). For any given elements $F$ and $F' \not \subseteq F$ of $\mathcal{F}$, we can verify that the above system of inequalities has no solution, in time that is a polynomial function of $|\mathcal{T}|$ and $C_{\mathrm{una}}(\Phi)$. Since there are polynomially many choices for $F$ and $F'$, we can verify in polynomial time that $\mathrm{Conv}(i(F))$ is a face of $\partial B_{\bar{x}}$ for each $F \in \mathcal{F}$.  

\textbf{Step 6: Certifying that the list of faces of $\partial B_{\bar{x}}$ is complete.}\\
The list of top-dimensional faces used to construct the space $P$ is the complete list of top-dimensional faces of $\partial B_{\bar{x}}$, otherwise the inclusion map $i$ would not have been surjective.

For every face $F \in \mathcal{F}$ there are top-dimensional faces $F_1, \cdots, F_r \in F$ with $r \leq d$ such that 
\[ F = \bigcap_{i=1}^{r} F_i, \]
where we have considered faces as subsets of the vertices $V$. Moreover, any intersection as above determines a face. Hence, we may go over all subsets of size at most $d$ of the set of top-dimensional faces, and verify that our list of faces is complete.
\end{proof}

\begin{remark}
Note that in defining the complexity of the basis for cohomology, we used $C_{\mathrm{una}}$ rather $C_{\mathrm{dig}}$. Although this was enough for the current application, it would be interesting to know if \textsc{Thurston norm ball for 3-manifolds with $b_1 \leq B$} still lies in \textbf{FNP} if we change the definition of the complexity of the cohomology basis to $C_{\mathrm{dig}}$. \end{remark}

\section{Decomposing a triangulated manifold along spheres and discs}
\label{Sec:SpheresDiscs}


In the previous section, we proved Theorem \ref{Thm:MainThmIrred}. This established Theorem \ref{thurston ball} for irreducible and boundary-irreducible 3-manifolds. In this section, we start to tackle the general case, by decomposing our given 3-manifold along spheres and discs.

The following is Theorem 11.4 and Addendum 11.5 of \cite{lackenby2016efficient}. It provides a method for building a triangulation of the connected summands of a 3-manifold $X$. The input is a triangulation $\mathcal{T}$ of $X$, together with normal spheres $S$ that specify the connected sum. The running time of the algorithm is bounded above in terms of the weight $w(S)$. Recall that this is the number of intersection points between $S$ and the 1-skeleton of $\mathcal{T}$. 

\begin{notation}
	For a metric space $X$ and $A \subset X$, denote the metric completion of $X - A$ with the induced path metric by $X \setminus \setminus A$. 
\end{notation}

\begin{thm}
\label{Thm:DecomposeAlongSpheres}
There is an algorithm that takes, as its input, the following data:
\begin{enumerate}
\item a triangulation $\mathcal{T}$ with $t$ tetrahedra of a compact orientable 3-manifold $X$;
\item a vector $(S)$ for a normal surface $S$ in $\mathcal{T}$ that is a union of disjoint spheres;
\item a simplicial $1$-cocycle $c$ on $\mathcal{T}$. 
\end{enumerate}
The output is a triangulation $\mathcal{T}'$ of $X \cut S$ and a simplicial 1-cocycle $c'$ on $\mathcal{T}'$ with the following properties:
\begin{enumerate}
\item the number of tetrahedra in $\mathcal{T}'$ is at most $200t$;
\item the classes $[c']$ and $i^\ast([c])$ in $H^1(X \cut S)$ are equal, where $i \colon X \cut S \rightarrow X$ is the inclusion map;
\item $C_{\mathrm{una}}(c') \leq 1200 t \, C_{\mathrm{una}}(c).$
\end{enumerate}
The algorithm runs in time that is bounded above by a polynomial function of 
$$t (\log w(S)) (\log (C_{\mathrm{una}}(c) + 1)).$$
\end{thm}

We need a small extension of this result.

\begin{thm}
\label{Thm:SpheresAndDiscs}
Theorem \ref{Thm:DecomposeAlongSpheres} remains true if $S$ is a union of disjoint spheres and discs.
\end{thm}

The proof is essentially identical, and we therefore only sketch it. 

When $S$ is a normal surface properly embedded in a compact orientable 3-manifold $X$ with a triangulation $\mathcal{T}$, then $X \cut S$ inherits a handle structure, as follows. One first dualises $\mathcal{T}$ to form a handle structure $\mathcal{H}$ for $X$. The normal surface $S$ then determines a surface that is standard in $\mathcal{H}$, which means that it is disjoint from the 3-handles, it intersects each handle in discs and, in the case of a 1-handle or 2-handle, these discs respect the handle's product structure. Then, by cutting along this surface, each $i$-handle of $\mathcal{H}$ is decomposed into $i$-handles in the required handle structure. We call this the \emph{induced} handle structure on $X \cut S$.

We do not actually construct this handle structure in the proof of Theorem \ref{Thm:DecomposeAlongSpheres}. The reason is that the number of handles (when $S$ is closed) is at least $w(S)$. So it is not possible to build this handle structure in time that is bounded above by a polynomial function of $\log w(S)$. 
In the next definition, it is useful to think about $\mathcal{H}'$ as the induced handle structure on $X' = X \cut S$ and where $S'$ is the copies of $S$ in $\partial X'$.

\begin{definition}

Let $\mathcal{H}'$ be a handle structure for a compact 3-manifold $X'$. Let $S'$ be a compact subsurface of $\partial X'$ such that $\partial S'$ is disjoint from the 2-handles and respects the product structure on the 1-handles. A handle $H$ of $\mathcal{H}'$ is a \emph{parallelity handle} for $(X',S')$ if
it admits a product structure $D^2 \times I$ such that
\begin{enumerate}
\item $D^2 \times \partial I = H \cap S'$;
\item each component of intersection between $H$ and another handle of $\mathcal{H}'$ is $\beta \times I$, for an arc $\beta$ in $\partial D^2$.
\end{enumerate}
The union of the parallelity handles is the \emph{parallelity bundle}.
\end{definition}

We will typically view the product structure $D^2 \times I$ on a parallelity handle as an $I$-bundle over $D^2$. It is shown in Lemma 5.3 of \cite{lackenby2009crossing} that these $I$-bundle structures patch together to form an $I$-bundle structure on the parallelity bundle.

\begin{definition}
\label{Def:BundleDefinitions}
Let $\mathcal{B}$ be an $I$-bundle over a compact surface $F$. Its \emph{horizontal boundary} $\partial_h \mathcal{B}$ is the $(\partial I)$-bundle over $F$. Its \emph{vertical boundary} $\partial_v \mathcal{B}$ is the $I$-bundle over $\partial F$. We say that a subset of $\mathcal{B}$ is \emph{vertical} if it is a union of fibres, and that it is \emph{horizontal} if it is a surface transverse to the fibres.
\end{definition}

The main step in the proof of Theorem 11.4 in \cite{lackenby2016efficient} was an application of the following result (Theorem 9.3 in \cite{lackenby2016efficient}).

\begin{thm}
There is an algorithm that takes, as its input, 
\begin{enumerate}
\item a triangulation $\mathcal{T}$, with $t$ tetrahedra, for a compact orientable 3-manifold $X$;
\item a vector $(S)$ for an orientable normal surface $S$;
\end{enumerate}
and provides as its output, the following data. If $S'$ is the two copies of $S$ in $\partial (X \cut S)$,
and $\mathcal{B}$ is the parallelity bundle for the pair $(X \cut S,S')$ with its induced handle structure, then
the algorithm produces a handle structure for $(X \cut S) \cut \mathcal{B}$ and, for each
component $B$ of $\mathcal{B}$, it determines:
\begin{enumerate}
\item the genus and number of boundary components of its base surface;
\item whether $B$ is a product or twisted $I$-bundle; and
\item for each component $A$ of $\partial_vB$, the location of $A$ in $(X \cut S) \cut \mathcal{B}$.
\end{enumerate}
It runs in time that is bounded  by a polynomial in $t \log(w(S))$.
\end{thm}

In the above, the meaning of the \emph{location} of $A$ is as follows. The intersection between $A$ and each handle of $(X \cut S) \cut \mathcal{B}$ is a union of fibres in the $I$-bundle structure on $A$, and hence is a copy of $I \times I$. In the case when $A$ lies entirely in $(X \cut S) \cut \mathcal{B}$, then $A$ is a union of these copies of $I \times I$, and in this case, the algorithm provides these copies of $I \times I$ in the order they appear as one travels around $A$. However, $A$ need not lie entirely in $(X \cut S) \cut \mathcal{B}$. This arises in the situation where $S$ has boundary. For example, if $D$ and $D'$ are normally parallel discs of $S$ that are incident to the boundary of $X$, then the space between them becomes a parallelity handle $D^2 \times I$ such that $\partial D^2 \times I$ intersects $\partial X$. Thus, in this situation, $A$ is decomposed into a union of copies of $I \times I$, which are the components of intersection between $A$ and the handles of $(X \cut S) \cut \mathcal{B}$ and also components of intersection between $A$ and $\partial X$. The algorithm provides the copies of $I \times I$ lying in $(X \cut S) \cut \mathcal{B}$ in the order they appear as one travels around $A$.

Thus, the triangulation $\mathcal{T}'$ is constructed by decomposing each of the handles of $(X \cut S) \cut \mathcal{B}$ into tetrahedra and by giving a compatible triangulation of $\mathcal{B}$. The number of handles of $(X \cut S) \cut \mathcal{B}$ is bounded above by a linear function of $t$ and each of these handles can intersect its neighbours in a very limited number of possibilities. Thus, it is not hard to triangulate $(X \cut S) \cut \mathcal{B}$ using at most $100t$ tetrahedra. In addition, we may ensure that the intersection with $\partial_v \mathcal{B}$ is simplicial. The horizontal boundary of each component $B$ of $\mathcal{B}$ is a planar surface, since $S$ is a union of spheres and discs. Thus, the topology of $B$ is determined entirely by the number of boundary components of its base surface and whether it is a twisted $I$-bundle or a product. It is shown that the total number of boundary components of the base surface of $\mathcal{B}$ is at most $10t$. Hence, it is not hard to construct the triangulation on $\mathcal{B}$ with at most $100t$ tetrahedra.

We now explain briefly how the cocycle $c'$ is constructed. This is explained in Addendum 11.5 in \cite{lackenby2016efficient}. 

For each oriented edge $e$ in $\mathcal{T}'$, we need to define $c'(e)$. It is convenient to dualise $c$ to form an oriented surface $F$ properly embedded in $X$. We may assume that $F$ is transverse to $S$ and that the intersection between $F$ and $\mathcal{B}$ is vertical in $\mathcal{B}$. If $e$ lies in $(X \cut S)  \cut \mathcal{B}$, then we define $c'(e)$ to be the algebraic intersection number between $e$ and $F \cut S$. This therefore defines the restriction of $c'$ to $\partial_v \mathcal{B}$. In the proof of Addendum 11.5 in \cite{lackenby2016efficient}, we replace $F$ by any compact oriented surface $F'$ that equals $F$ in $(X \cut S)  \cut \mathcal{B}$, that is vertical in $\mathcal{B}$ and that satisfies $\partial_v \mathcal{B} \cap F' = \partial_v \mathcal{B} \cap F$. It is shown how to do this while maintaining control over the number of intersections with the edges of $\mathcal{T}'$. In particular, the cocycle $c'$ dual to $F'$ satisfies $C_{\mathrm{una}}(c') \leq 1200 t \, C_{\mathrm{una}}(c)$. Now, $F'$ and $F$ differ by a class that is represented by a vertical surface in $\mathcal{B}$ disjoint from $\partial_v \mathcal{B}$. In our situation, any such surface is dual to the trivial class in $H^1(X \cut S)$, since $S$ is a union of spheres and discs. Thus, in fact, $[c']$ and $i^\ast([c])$ are equal. 

This completes the outline of the proof of Theorem \ref{Thm:SpheresAndDiscs}. We will first apply it to essential spheres in $X$ with the following property.

\begin{definition}
A collection of disjoint essential spheres $S$ properly embedded in a 3-manifold $X$ is \emph{complete} if the manifold obtained from $X \cut S$ by attaching a 3-ball to each spherical boundary component is irreducible.
\end{definition}



The following was proved by King (Lemma 4 in \cite{King}). King's result is stated for closed orientable 3-manifolds, but his argument extends immediately to compact orientable 3-manifolds with boundary. (See also Lemma 2.6 in \cite{mijatovic}).

\begin{thm}
Let $\mathcal{T}$ be a triangulation of a compact orientable 3-manifold $X$ with $t$ tetrahedra. Then there is a complete collection of disjoint essential normal spheres in $\mathcal{T}$ with weight at most $2^{185t^2}$.
\label{king}
\end{thm}

It might be possible to improve this estimate. It was shown by Jaco and Tollefson (Theorem 5.2 in \cite{JacoTollefson}) that, when $X$ is closed, it contains a complete collection of essential spheres, each of which is a vertex normal surface. (See Section \ref{Sec:Normal} for the definition of a vertex normal surface.) By Lemma 3.2 in \cite{HassLagarias} a vertex normal surface has weight at most $28t2^{7t-1}$. However, the generalisation of Jaco and Tollefson's argument to manifolds with non-empty boundary does not seem so straightforward. In any case, Theorem \ref{king} is sufficient for our purposes.

Jaco and Tollefson also proved the following result dealing with compression discs for the boundary (Theorem 6.2 in \cite{JacoTollefson}). It refers to a \emph{complete} collection of compressing discs, which means that the manifold obtained by compressing along these discs has incompressible boundary.

\begin{thm}
\label{Thm:JacoTollefsonDiscs}
Let $\mathcal{T}$ be a triangulation of a compact orientable irreducible 3-manifold $X$. Then $X$ has a complete collection of disjoint compressing discs, each of which is a vertex normal surface. Hence, each such disc has weight at most $28t2^{7t-1}$, and their total weight is at most $280t^2 2^{7t-1}$.
\end{thm}

The final estimate is a consequence of the well known result, essentially due to Kneser \cite{kneser1929geschlossene}, that in any collection of more than $10t$ disjoint normal surfaces, at least two of the surfaces are parallel.



\begin{proof}[Proof of Theorem \ref{thurston ball}]
We are given a triangulation $\mathcal{T}$ of the compact orientable 3-manifold $X$ and a collection of integral simplicial cocycles $\phi_1, \cdots, \phi_b$ that forms a basis for $H^1(X; \mathbb{R})$. Our goal is to compute the Thurston norm ball. Recall that the required output is:
\begin{enumerate}
\item A collection of elements that are integral linear combinations of $\phi_1, \cdots, \phi_b$. These will form a basis $\mathcal{B}$ for the subspace $W$ of $H^1(X;\mathbb{R})$ with Thurston norm zero.
\item A collection $V$ of rational linear combinations of $\phi_1, \cdots, \phi_b$ that project to the vertices of the norm ball in $H^1(X; \mathbb{R}) / W$.
\item A collection $\mathcal{F}$ of subsets of $V$ that form the faces.
\end{enumerate}
These will all be part of our certificate. In addition, the following will also form our certificate:
\begin{enumerate}
\item A normal surface $S$ in $\mathcal{T}$, given via its vector $(S)$, that is in fact a complete collection of disjoint essential spheres. It has weight at most $2^{185t^2}$ where $t = |\mathcal{T}|$.
\item A triangulation $\mathcal{T}'$ for the manifold $X'$ obtained by cutting along $S$ and then attaching a 3-ball to each spherical boundary component.
\item A collection of simplicial 1-cocycles $\phi'_1, \cdots, \phi'_b$ that are the images of $\phi_1, \cdots, \phi_b$ in $H^1(X')$ under the map $H^1(X) \rightarrow H^1(X \cut S) \cong H^1(X')$.
\item A normal surface $D$ in $\mathcal{T}'$, given via its vector $(D)$, that is in fact a complete collection of disjoint compression discs for $\partial X'$. It has weight at most $280 |\mathcal{T}'|^2 2^{7|\mathcal{T}'|-1}$.
\item A triangulation $\mathcal{T}''$ for $X'' = X' \cut D$.
\item A collection of simplicial 1-cocycles $\phi''_1, \cdots, \phi''_b$ that are the images of $\phi'_1, \cdots, \phi'_b$ in $H^1(X'')$.
\item A certificate for the decision problem \textsc{Thurton norm ball for irreducible boundary-irreducible 3-manifolds with $b_1 \leq B$}, which provides the data for the Thurston norm ball of $H^1(X'')$. This data is a basis for the subspace $W''$ of $H^1(X''; \mathbb{R})$ with Thurston norm zero, together with the vertices $V''$ and faces $\mathcal{F}''$ for the norm ball in $H^1(X''; \mathbb{R}) / W''$.
\end{enumerate}

The certificate is verified as follows:
\begin{enumerate}
\item Verification that $S$ is a collection of spheres using the algorithm in \cite{agol2006computational}.
\item Verification that $\mathcal{T}'$ is a triangulation of $X'$ and that $\phi'_1, \cdots, \phi'_b$ are the images of $\phi_1, \cdots, \phi_b$ in $H^1(X')$, using Theorem \ref{Thm:DecomposeAlongSpheres}.
\item Verification that $D$ is a collection of discs using \cite{agol2006computational}.
\item Verification that $\mathcal{T}''$ is a triangulation of $X''$ and that $\phi''_1, \cdots, \phi''_b$ are the images of $\phi'_1, \cdots, \phi'_b$ in $H^1(X'')$ using Theorem \ref{Thm:DecomposeAlongSpheres}.
\item Verification that each component of $X''$ either is irreducible and boundary-irreducible or is a rational homology 3-sphere, using Corollary \ref{Cor:IrredIncompNP}.
\item Verification of the certificate for \textsc{Thurton norm ball for irreducible boundary-irreducible 3-manifolds with $b_1 \leq B$} for the manifold $X''$. Any component of 
$X''$ that is a (possibly reducible) rational homology 3-sphere has trivial Thurston norm, and hence plays no role here. 
\item Verification that we may write $\mathcal{B} = \mathcal{B}_1 \cup \mathcal{B}_2 \cup \mathcal{B}_3$ such that \\
i) the elements of $\mathcal{B}_1$ form a basis for the kernel of the map 
\[H^1(X) \rightarrow H^1(X - N^\circ(S)) \cong H^1(X'),\] 
where $N(S)$ is a tubular neighbourhood of $S$;\\
ii) the elements of $\mathcal{B}_2$ project to a basis for the kernel of the map
\[H^1(X') \rightarrow H^1(X' \cut D) = H^1(X''),\] 
and this projection is one-to-one;\\
iii) the elements of $\mathcal{B}_3$ project to a basis of $W''$ and this projection is one-to-one.
\item Verification that the map $H^1(X) \rightarrow H^1(X'')$ sets up a bijection $V \rightarrow V''$ and a bijection $\mathcal{F} \rightarrow \mathcal{F}''$.
\end{enumerate}

The input to \textsc{Thurton norm ball for irreducible boundary-irreducible 3-manifolds with $b_1 \leq B$} requires a collection of integral cocycles that forms a basis for $H^1(X''; \mathbb{R})$. Although $\phi''_1, \cdots, \phi''_b$ might not form a basis, they do form a spanning set, and therefore some subset of them (which can easily be found) forms a basis.

The output provides integral cocycles that form a basis for the subspace $W''$ of norm zero. It also consists of a set of points $V''$ in $H^1(X''; \mathbb{Q})$ that give the vertices of the norm ball and a collection $\mathcal{F}''$ of subsets of $V''$ that give the faces. Looking at the long exact sequence of the pair $(X, X - N^\circ(S))$ we have
\begin{eqnarray*}
H^1(X, X- N^\circ(S)) \rightarrow H^1(X) \rightarrow H^1(X - N ^\circ(S)) \rightarrow H^2(X, X - N^\circ(S)) \rightarrow \cdots
\end{eqnarray*}
By excision and the Poincar\'{e} duality we have
\[ H^1(X , X - N^\circ(S)) \cong H^1(N(S), \partial N(S)) \cong H_2(N(S)) \cong H_2(S). \]
Similarly $H^2(X , X - N^\circ(S)) \cong H_1(S) \cong 0$. Therefore, the above long exact sequence takes the form
\[ H_2(S) \rightarrow H^1(X) \xrightarrow{p}  H^1(X - N ^\circ(S)) \rightarrow 0. \]
Thus, $p$ is surjective. It is Thurston norm-preserving and its kernel is generated by spheres in $S$. Similarly, the map $H^1(X') \rightarrow H^1(X'')$ is surjective, norm-preserving and its kernel is generated by discs in $D$. Thus, we let $\mathcal{B}_1$ be a basis for the kernel of $p$. We let $\mathcal{B}_2$ be a collection of elements that are sent by $p$ to a basis for the kernel of $H^1(X') \rightarrow H^1(X'')$. Finally, assume that $\mathcal{B}_3$ is a subset of $H^1(X)$ that projects to a basis for the subspace $W''$ of $H^1(X''; \mathbb{R})$ with Thurston norm zero. Then $\mathcal{B} = \mathcal{B}_1 \cup \mathcal{B}_2 \cup \mathcal{B}_3$ is a basis for the subspace $W$ of $H^1(X;\mathbb{R})$ with Thurston norm zero. 
Now, there is an induced isomorphism from $H^1(X; \mathbb{R}) /W$ to $H^1(X'' ;\mathbb{R}) / W''$ which is norm-preserving. Thus, we may obtain the points $V$ in $H^1(X ; \mathbb{Q})$ by running through each element of $V''$ in $H^1(X''; \mathbb{Q})$ and picking a point in its inverse image. A set of points in $V$ spans a face if and only if the corresponding points in $V''$ do. Thus, we obtain the required output for \textsc{Thurton norm ball for 3-manifolds with $b_1 \leq B$}.

We need to show that the certificate exists and can be verified in polynomial time.


By Theorem \ref{king}, there is a complete collection of disjoint essential normal spheres, $S$, in $\mathcal{T}$ with weight at most $2^{185t^2}$ where $t = |\mathcal{T}|$, the number of tetrahedra in $\mathcal{T}$.  The normal coordinates of elements of $S$ are part of the certificate, and are given to us non-deterministically. Now we may decompose the manifold along $S$ and then attach balls to any resulting spherical boundary components. Let $X'$ be the resulting irreducible 3-manifold. Theorem \ref{Thm:DecomposeAlongSpheres} guarantees that we may build a triangulation $\mathcal{T}'$ of $X'$ with no more than $O(|\mathcal{T}|)$ tetrahedra, and simplicial 1-cocycles $\phi_j' \in H^1(X' ; \mathbb{Z})$ such that the cohomology classes $i^*([\phi_j])$ and $[\phi_j']$ are equal and $C_{\mathrm{una}}(\phi_j ')$ is bounded above by a polynomial function of $|\mathcal{T}|$ and $C_{\mathrm{una}}(\phi_j)$. Moreover, this procedure can be done in time that is a polynomial function of $bt ( \log w(S) )(\log(C_{\mathrm{una}}(\phi_j))+1))$, which is bounded above by a polynomial function of $|\mathcal{T}|$ and $C_{\mathrm{una}}(\Phi)$ by our assumption on the weight of $S$ and the complexity of the homology basis. 

  

By Theorem \ref{Thm:JacoTollefsonDiscs}, there is a complete collection of compression discs for $X'$ that are normal in $\mathcal{T}'$ and with weight at most $280 |\mathcal{T}'|^2 2^{7|\mathcal{T}'|-1}$. Applying Theorem \ref{Thm:SpheresAndDiscs}, we may cut along these discs, forming a 3-manifold $X''$ and obtain a triangulation $\mathcal{T}''$ and cocycles $\phi''_1, \cdots, \phi''_b$. As above, the number of tetrahedra is $O(|\mathcal{T}'|)$ and therefore $O(|\mathcal{T}|)$. The cocycles $\phi_j''$ have $C_{\mathrm{una}}$ that is bounded above by a polynomial function of $|\mathcal{T}|$ and $C_{\mathrm{una}}(\Phi)$. The procedure may be completed in polynomial time.

Finally, the certificate for \textsc{Thurton norm ball for irreducible boundary-irreducible 3-manifolds with $b_1 \leq B$} is verified in polynomial time.
\end{proof}

\section{Other representations of the manifold and knot}

In the decision problem \textsc{Knot genus in the fixed 3-manifold $M$}, the manifold $M$ is given to us by means of a diagram $D$ for $\Gamma \cup L$, where $\Gamma$ is a link in $S^3$ and $L$ is a framed link, and $K$ is specified by giving a diagram for $K \cup \Gamma \cup L$ that contains $D$ as a subdiagram. This method of representing $M$ and $K$ is a natural one. However, it also played a critical role in the proof of Theorem \ref{main:boundary}, as the construction of an efficient basis for $H_2(M - N^\circ(K), \partial M \cup \partial N(K))$ relied on this presentation of $M$ and $K$. So it is reasonable to consider other methods for representing $M$ and $K$, and to ask whether the resulting decision problems still lie in \textbf{co-NP}. 

For simplicity, we will focus on closed orientable 3-manifolds $M$, although much of our discussion does generalise to the case of non-empty toroidal boundary.

One way of specifying a closed orientable 3-manifold is by giving a Heegaard splitting for it. Here, we are given a closed orientable surface $S$, a union $\alpha$ of disjoint simple closed curves $\alpha_1, \cdots, \alpha_g$ in $S$ and another collection $\beta$ of disjoint simple closed curves $\beta_1, \cdots, \beta_g$ in $S$, with the property that $S - N^\circ(\alpha)$ and $S - N^\circ (\beta)$ are both planar and connected. We also assume that each component of $S - N^\circ(\alpha \cup \beta)$ is a disc. We suppose that $M$ is obtained by attaching two handlebodies to $S$ so that the curves $\alpha$ bound discs in one handlebody and the curves $\beta$ bound discs in the other handlebody. We think of this presentation of $M$ as fixed and given to us in some way, for example by specifying a triangulation of $S$ in which the curves $\alpha$ and $\beta$ are all simplicial.

We now wish to add $K$ to the picture. We do this by specifying a diagram for $K$ in $S$, in other words an immersed curve with generic double points at which under/over crossing information is specified. We also assume that this immersed curve intersects the $\alpha$ and $\beta$ curves transversely. We call this a \emph{diagram} for $K$. This specifies an embedding of $K$ into $S \times [-1,1]$ and hence into $M$, once we have agreed on the convention that the handlebody with discs attached to the $\alpha$ curves lies on the $S \times \{ - 1\}$ side. We say that the \emph{total crossing number} of $K$ is the sum of the number of crossings of $K$ with itself and its number of intersections with the $\alpha$ and $\beta$ curves. This is our measure of complexity for $K$.

Note that every knot $K$ in $M$ is specified by such a diagram, as follows. Each handlebody is a regular neighbourhood of a graph. We can isotope $K$ off a small open regular neighbourhood of these two graphs. It then lies in the complement of this open neighbourhood, which is a copy of $S \times [-1,1]$. The projection $S \times [-1,1] \rightarrow S$ onto the first factor specifies the diagrammatic projection map. After a small isotopy, the image of $K$ has only generic double point singularities, which form the crossings of $K$ with itself.

Thus, we can phrase the following decision problem. We fix a Heegaard diagram for $M$ in a closed orientable surface $S$, as above.

\medskip 
\noindent \textbf{Problem}: \textsc{Knot genus in the fixed closed orientable 3-manifold $M$ via a Heegaard diagram}.\\
\emph{Input}: A diagram of $K$ in $S$, as above, and an integer $g \geq 0$ in binary.\\
\emph{Input size}: The total crossing number of $K$ plus the number of digits of $g$ in binary.\\ 
\emph{Question}: Is the genus of $K$ less than or equal to $g$?\\

\begin{thm}
	\label{Thm:HeegaardNP}
	\textsc{Knot genus in the fixed closed orientable 3-manifold $M$ via a Heegaard diagram} lies in \textbf{co-NP}.
\end{thm}

\begin{remark}
	\label{Rem:NonDisc}
	We briefly discuss the above requirement that each component of $S - N^\circ(\alpha \cup \beta)$ is a disc. This almost always happens automatically anyway. Indeed, if some component of $S - N^\circ(\alpha \cup \beta)$ is not a disc, then it contains an essential simple closed curve that bounds a disc in both handlebodies. The Heegaard splitting is then reducible. However, we can always ensure that each component of $S - N^\circ(\alpha \cup \beta)$ is a disc, by performing an isotopy to $\beta$. For if $S - N^\circ(\alpha \cup \beta)$ is not a union of discs, then we can pick a properly embedded essential arc in some component joining the $\beta$ curves to the $\alpha$ curves, and then isotope the relevant $\beta$ curve along it, to introduce two new intersection points between the $\alpha$ curves and the $\beta$ curves. We call this a \emph{finger move}. Repeating this process if necessary, we end with the required collection of $\alpha$ and $\beta$ curves.
	
	The reason for making this requirement is that it avoids the following scenario. Suppose that some component $P$ of $S - N^\circ(\alpha \cup \beta)$ is not a disc. Then we could choose a diagram of some knot $K$ to wind many times around $P$, plus possibly intersect $\partial P$. In this way, we would get infinitely many distinct diagrams, all with the same total crossing number. Thus, in this case, the total crossing number would not become a reasonable measure for the complexity of the diagram.
\end{remark}

We will prove Theorem \ref{Thm:HeegaardNP} by reducing \textsc{Knot genus in the fixed closed orientable 3-manifold $M$ via a Heegaard diagram} to \textsc{Knot genus in the fixed closed orientable 3-manifold $M$}. In order to this, we need an algorithm to translate a diagram for a knot $K$ in a Heegaard surface to a planar diagram for $K$ lying in the complement of some surgery curves. This is provided by the following result.

\begin{thm}
	\label{Thm:HeegaardToPlanarDiagram}
	Let $S$ be a closed orientable surface with curves $\alpha = \alpha_1 \cup \cdots \cup \alpha_g$ and $\beta = \beta_1 \cup \cdots \cup \beta_g$ specifying a Heegaard splitting of $M$. Suppose that $S - N^\circ(\alpha \cup \beta)$ is a union of discs. Then there is a diagram $D$ of a framed link $L$ in $S^3$ that specifies a surgery description of $M$ and that has the following property. Let $K$ be a knot in $M$ given via a diagram of $K$ in $S$ with total crossing number $c$. Then there is a diagram of a knot in the complement of $L$ that is isotopic to $K$, that contains $D$ as a subdiagram and that has total crossing number $O(c^2)$. This may be constructed in polynomial time as a function of $c$. Here, the implied constant depends only on $M$ and the Heegaard splitting, and not on $K$.
\end{thm}

We start with the case of the \emph{standard} Heegaard splitting for $S^3$. This has curves $\alpha_1, \cdots, \alpha_g$ and $\beta_1, \cdots, \beta_g$ satisfying $|\alpha_i \cap \beta_j| = \delta_{ij}$.

\begin{lem}
	\label{Lem:StandardHeegaardToPlanarDiagram}
	Let $S$ be a closed orientable surface with genus $g$, equipped with curves that give the standard genus $g$ Heegaard splitting for the 3-sphere. Let $K$ be a knot given by a diagram in $S$ with total crossing number $c$. Then there is a diagram for $K$ in the plane with crossing number at most $c^2$. This may be constructed in polynomial time as a function of $c$. This remains true if $K$ is a link with several components. Furthermore, some of its components may be framed via surface framing in $S$, in which case we can also require that the resulting planar diagram specifies the same framing on these components.
\end{lem}

\begin{proof}
	Let $c_K$ be the number of crossings in $S$ between $K$ and itself. Then the total crossing number $c$ of $K$ is
	\[ c_K + \sum_i |K \cap \alpha_i| + \sum_i |K \cap \beta_i|. \]
	We will modify the given diagram of $K$ in $S$ so that it becomes disjoint from the $\alpha$ curves. So consider a curve $\alpha_i$. We may isotope its intersection points with $K$ so that they all lie in a small neighbourhood of the point $\alpha_i \cap \beta_i$. We may then isotope $K$ across the disc bounded by $\beta_i$. This has the effect of removing these points of $\alpha_i \cap K$, but possibly introducing new crossings of $K$. Near each point of $K \cap \beta_i$, we get $|\alpha_i \cap K|$ new crossings of $K$. Thus, after these modifications, the number of crossings between $K$ and itself is 
	\[ c_K + \sum_i |K \cap \alpha_i| \cdot |K \cap \beta_i| \]
	which is clearly at most $c^2$.
	We now use this to create a diagram for $K$ in the plane. We compress $S$ along the curves $\alpha_1, \cdots, \alpha_g$. Since the diagram for $K$ is now disjoint from these curves, the result is a diagram for $K$ in the 2-sphere, and hence the plane.
\end{proof}

We now extend this to slightly more general Heegaard splittings for $S^3$.

\begin{lem}
	\label{Lem:Isotopy}
	Let $S$ be a closed orientable surface with genus $g$. Let $\alpha = \alpha_1 \cup \cdots \cup \alpha_g$ be a union of disjoint simple closed curves that cut $S$ to a planar connected surface. Let $\beta = \beta_1 \cup \cdots \cup \beta_g$ be another collection of disjoint simple closed curves with the same property. Suppose that there is an isotopy taking $\beta$ to curves that, with $\alpha$, form the standard Heegaard splitting for $S^3$. Let $K$ be a knot given by a diagram in $S$ with total crossing number $c$. Then there is a planar diagram for $K$ in $S^3$ with crossing number at most $c^2$. This diagram may be constructed in a polynomial time as a function of $c$. Here, the implied constants depend on the curves $\alpha$ and $\beta$ but not $K$. This remains true if $K$ is a link with several components, some of which may be framed.
\end{lem}

\begin{proof}
	We are assuming that there is an isotopy taking $\beta_1, \cdots, \beta_g$ to curves $\beta'_1, \cdots, \beta'_g$ satisfying $|\alpha_i \cap \beta'_j| = \delta_{ij}$. This isotopy may be performed by performing a sequence of \emph{bigon} moves; see for example Proposition 1.7 of \cite{farb2011primer}. Here, one has a disc $D$ in $S$ with the interior of $D$ disjoint from $\alpha$ and $\beta$, and with $\partial D$ consisting of a sub-arc of an $\alpha$ curve and a sub-arc of a $\beta$ curve. The isotopy slides this $\beta$ arc across $D$. We shall show how to create a new diagram for $K$ in $S$ when such a move is performed. This will have the property that the total crossing number of the new diagram is at most the total crossing number of the old diagram. Hence, after these moves are performed, we may construct a diagram for $K$ in the plane with crossing number at most $c^2$, using Lemma \ref{Lem:StandardHeegaardToPlanarDiagram}.
	
	Within the disc $D$, there is a portion of the diagram for $K$. We will pull this portion of the diagram entirely through $\alpha$ or through $\beta$, so that after this, the arcs of $K$ within $D$ run directly from $\alpha$ to $\beta$ without any crossings. The choice of whether to slide this portion of the diagram through $\alpha$ or $\beta$ is made so that it does not increase the number of crossings. Thus, if there are $c_\alpha$ crossings between $K$ and $\alpha$ along $\partial D$, and $c_\beta$ crossings between $K$ and $\beta$ along $\partial D$, then after this operation, the number of crossings between $K$ and $\partial D$ is $2 \min \{ c_\alpha, c_\beta \}$. Thus, the total crossing number of $K$ has not gone up. After this, we may isotope $\beta$ across $D$ without changing the number of crossings.
\end{proof}

\begin{lem}
	\label{Lem:DehnTwist}
	Let $S$ be a closed orientable surface with genus $g$. Let $\alpha$ be disjoint simple closed curves that cut $S$ to a planar connected surface. Let $\beta$ be another collection of disjoint simple closed curves with the same property. Suppose that each component of $S - N^\circ(\alpha \cup \beta)$ is a disc. Let $C$ be an essential simple closed curve in $S$. Then there is a constant $\lambda \geq 1$ with the following property. Let $K$ be a link, some components of which may be framed, given by a diagram in $S$ with total crossing number $c$. Let $K'$ be obtained from $K$ by Dehn twisting about $C$, and let $\beta'$ also be  obtained from $\beta$ by Dehn twisting about $C$. Then the total crossing number of the diagram on $S$ given by $K' \cup C$ with respect to the curves $\alpha$ and $\beta'$ is at most $\lambda c + \lambda$. Moreover, this diagram may be constructed in polynomial time as a function of $c$.
\end{lem}

\begin{proof}
	By assumption, each component of $S - N^\circ(\alpha \cup \beta)$ is a disc. We realise this as a convex Euclidean polygon $P$ with straight sides, where each side is parallel to an arc of intersection with $\alpha$ or $\beta$. We may assume that $C$ intersects $\alpha \cup \beta$ minimally, and hence that its intersection with this disc consists of straight arcs. Pick a point $p$ in the interior of $P$ disjoint from $C$. Let $\epsilon P$ be the result of performing a dilation to $P$ based at $p$, with scale factor $\epsilon > 0$ small enough so that $\epsilon P$ is disjoint from $C$. We now isotope the diagram of $K$ within $P$, without changing the points of $K \cap \partial P$, as follows. We rescale the diagram using the dilation based at $p$ so that it lies within $\epsilon P$. Then in the annular region $P \cut \epsilon P$, we set the diagram to be a collection of disjoint straight arcs, each running from a point on $\partial P$ to a point on $\partial (\epsilon P)$ along a straight line that goes through $p$. Each intersection point between $K$ and $C$ lies in a straight arc of $K$, and this straight arc has an endpoint on $\alpha \cup \beta$. Thus, there is a constant $\lambda_1>0$, depending on $\alpha$, $\beta$ and $C$, such that the number of crossings between $K$ and $C$ is at most $\lambda_1 c$. We now perform the Dehn twist about $C$, giving the link $K'$ and the curves $\beta'$. The intersection points between $K'$ and $\beta'$ correspond to the intersection points between $K$ and $\beta$. The crossings of $K'$ with itself correspond to the crossings of $K$ with itself. Each crossing between $K$ and $C$ gives $|C \cap \alpha|$ extra crossings between $K'$ and $\alpha$. Thus, the total crossing number of $K$ goes up by a factor of at most $1+ \lambda_1 |C \cap \alpha|$. We also need to consider the crossings involving $C$. There are at most $\lambda_1 c$ of these with $K'$, and at most a constant number with $\alpha \cup \beta'$. The required bound then follows.
\end{proof}

\begin{remark}
	\label{Rem:ExtraFingerMoves}
	In the above lemma, we made the hypothesis that each component of $S - N^\circ(\alpha \cup \beta)$ is a disc. We would like to ensure that $\alpha$ and $\beta'$ have the same property, in other words that each component of $S - N^\circ(\alpha \cup \beta')$ is a disc. However, this might not be the case. Near $C$, there are various components of $S - N^\circ(\alpha \cup \beta)$. The components of $S - N^\circ(\alpha \cup \beta')$ are obtained by cutting along $C$ and then possibly gluing some of these together in a different way. An example is shown in Figure \ref{fig:dehn}, where a component of $S - N^\circ(\alpha \cup \beta')$ is obtained from two components of $S - N^\circ(\alpha \cup \beta \cup C)$ glued together. However, if this process does create some components of $S - N^\circ(\alpha \cup \beta')$ that are not discs, they may be cut into discs using finger moves, as in Remark \ref{Rem:NonDisc}. The number of finger moves that are needed is at most $|C \cap (\alpha \cup \beta)|$. This has the effect of increasing the total crossing number of $K' \cup C$ by at most $2|K \cap C|$, which is at most a constant times $c$.
\end{remark}

\begin{figure}[h]
\centering
\includegraphics[width=0.7\textwidth]{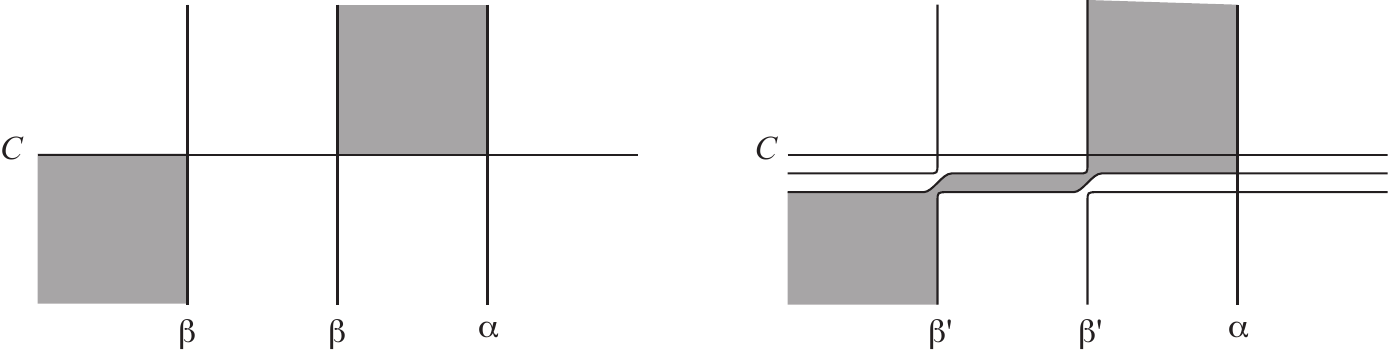}
\caption{Dehn twisting $\beta$ along $C$} \label{fig:dehn}
\end{figure}

\begin{proof}[Proof of Theorem \ref{Thm:HeegaardToPlanarDiagram}]
	We are given a closed orientable surface $S$ with curves $\alpha$ and $\beta$ specifying a fixed Heegaard splitting of $M$. We are also given a diagram in $S$ of a knot $K$ with total crossing number $c$. We will change the diagram and the Heegaard splitting in a sequence of modifications. There is an orientation-preserving homeomorphism of $S$ taking the curves $\beta = \beta_1 \cup \cdots \cup \beta_g$ to curves $\beta'' = \beta''_1 \cup \cdots \cup \beta''_g$ that satisfy $\alpha_i \cap \beta''_j = \delta_{ij}$. This homeomorphism is obtained by a product of Dehn twists about simple closed curves in $S$, followed by an isotopy. We can apply such a Dehn twist if we also add a surgery curve $C$ that undoes it. Thus, we can replace the knot $K$ and curves $ \beta$ by a knot $K'$ together with the framed surgery curve $C$, and the curves $\beta'$ obtained by Dehn twisting along $C$. By Lemma \ref{Lem:DehnTwist}, the new knot $K'$ and surgery curve $C'$ have total crossing number bounded above by a constant times $c$. (The additive constant in the lemma can be subsumed into the multiplicative constant since we can assume that the total crossing number is non-zero.) By Remark \ref{Rem:ExtraFingerMoves}, we can also ensure that each component of $S - N^\circ(\alpha \cup \beta')$ is a disc, at a cost of increasing the total crossing number of $K' \cup C$ by at most a constant factor. Repeating this for each Dehn twist in the sequence, we end with curves $\beta'$, a diagram for $K$ and the framed link $L$ specifying the surgery. This has total crossing number that is at most $O(c)$. The curves $\beta'$ are isotopic to $\beta''$, and so by Lemma \ref{Lem:Isotopy}, we obtain a planar diagram for $K \cup L$ with total crossing number that is at most $O(c^2)$.
\end{proof}

This completes the proof of Theorem \ref{Thm:HeegaardNP}.

\begin{remark} There is another possible way of representing $M$ and $K$ using triangulations. Fix a 3-manifold $M$ up to homeomorphism. We could be simply given a triangulation $\mathcal{T}$ and a knot $K$ as a subcomplex of $\mathcal{T}$, and we would be told that $\mathcal{T}$ was indeed a triangulation of $M$. But in the absence of an efficient method converting this triangulation to a \emph{fixed} triangulation of $M$, it is hard to see how this could be useful. 
\end{remark}

\bibliographystyle{plain}
\bibliography{knot-genus}

\end{document}